\documentclass[reqno]{amsart}
\usepackage{amssymb}
\usepackage{amsthm}
\usepackage[margin=1in]{geometry}
\usepackage[colorlinks=true, linkcolor=blue]{hyperref}

\title{Decay estimates for Schr\"{o}dinger's equation with magnetic potentials in three dimensions}
\author{Marius Beceanu}
\author{Hyun-Kyoung Kwon}

\newcommand{\mc}{\mathcal}
\newcommand{\R}{\mathbb R}
\newcommand{\C}{\mathbb C}
\newcommand{\K}{\mathcal K}
\newcommand{\KK}{{\mathcal K_2}}
\newcommand{\U}{\mathcal U}
\newcommand{\B}{\mathcal B}
\newcommand{\les}{\lesssim}
\newcommand{\dd}{\,d}
\newcommand{\be}{\begin{equation}}
\newcommand{\ee}{\end{equation}}
\newcommand{\lb}{\label}
\newcommand{\ds}{\displaystyle}
\DeclareMathOperator{\Div}{div}
\renewcommand{\div}{\Div}
\DeclareMathOperator{\Ker}{Ker}

\newcommand{\ov}{\overline}

\newtheorem{proposition}{Proposition}
\newtheorem{lemma}[proposition]{Lemma}
\newtheorem{theorem}[proposition]{Theorem}
\newtheorem{corollary}[proposition]{Corollary}
\newtheorem{assumption}{Assumption}
\newtheorem{definition}{Definition}


\begin{document}
\maketitle
\begin{abstract}
In this paper we prove that Schr\"{o}dinger's equation with a Hamiltonian of the form $H=-\Delta+i(A \nabla + \nabla A) + V$, which includes a magnetic potential $A$, has the same dispersive and solution decay properties as the free Schr\"{o}dinger equation. In particular, we prove $L^1 \to L^\infty$ decay and some related estimates for the wave equation.
	
The potentials $A$ and $V$ are short-range and $A$ has four derivatives, but they can be arbitrarily large. All results hold in three space dimensions.
\end{abstract}

\tableofcontents

\section{Introduction}
\subsection{Setup}

In three space dimensions, consider symmetric Hamiltonians of the form
$$
H=-\Delta+i(A\nabla+\nabla A)+V,
$$
where $-\Delta$ is the free Laplacian, $A:\R^3 \to \R^3$ is the magnetic potential (or gradient term in the potential), and $V:\R^3 \to \R$ is the electric (scalar) potential. When $A$ and $V$ are real-valued and $V \in \K$ (see below) and
$$
A \in \mc A = \{A \mid \sum_{j=0}^\infty 2^j \|A\|_{L^\infty(D_j)} < \infty\}
$$
where $D_0=B(0, 1)$, $D_n=B(0, 2^n) \setminus B(0, 2^{n-1})$ for $n \geq 1$ (see \cite{IonSch} for this and more general conditions), then $H$ is self-adjoint.


This paper regards the Schr\"{o}dinger equation (\ref{sch}) where the time evolution is driven by the Hamiltonian $H$, which contains a magnetic potential.

An important part in the analysis is played by the global Kato space introduced in \cite{RodSch}
$$
\K = \{V \mid \sup_{y \in \R^3} \int_{\R^3} \frac {|V(x)| \dd x}{|x-y|} < \infty\},\ \|V\|_{\K} = \sup_{y \in \R^3} \int_{\R^3} \frac {|V(x)| \dd x}{|x-y|},
$$
as well as similar spaces defined below such as $\K_{\log}$, $\KK$, $\K_{2, \log^2}$, and $L \log L$, some including logarithmic modifications, see definitions (\ref{llogl}-\ref{kklog2}). For convenience note that
$$
L^{3/2-\epsilon} \cap L^{3/2+\epsilon} \subset L^{3/2, 1} \subset \K,\ L^{3/2-\epsilon} \cap L^{3/2+\epsilon} \subset \K_{\log},\ L^{3-\epsilon} \cap L^{3+\epsilon} \subset L^{3, 1} \subset \KK,\ L^{3-\epsilon} \cap L^{3+\epsilon} \subset \K_{2, \log^2}
$$
and $L^1 \cap L^{1+\epsilon} \subset L \log_+ L \subset L^1$. Here $L^{3/2, 1}$ and $L^{3, 1}$ are Lorentz spaces, see \cite{BerLof}.

When $H$ is self-adjoint, functional calculus defines a family of commuting normal operators
$$
f(H) = \int_{\sigma(H)} f(\lambda) \dd E_\lambda
$$
that includes $H$ itself. This family also includes the propagator for equation (\ref{sch}) and several others.

\subsection{Main results}

Consider the following partial differential equation with magnetic potential:
\be\lb{sch}
i \partial_t f + H f = F,\ f(0)=f_0 \text{ (Schr\"{o}dinger's equation).}
\ee
The solutions are expressed by means of Duhamel's formula and propagators:
$$
f(t) = e^{itH} f_0 - i \int_0^t e^{i(t-s)H} F(s) \dd s.
$$
Here $e^{itH}$ is the propagator for Schr\"{o}dinger's equation (\ref{sch}).

Let
\be\lb{X}
A \in X= \{A \in \K_{2, \log^2} \cap \K_{\log} \cap L^{3, 1} \cap L^{3/2, 1} \mid \nabla A \in \K_{\log} \cap \K_{2, \log^2} \cap L^{3/2, 1},\ \nabla^2 A \in \K_{2, \log^2} \cap L^1,\ \nabla^3 A, \nabla^4 A \in L \log L\}
\ee
and
\be\lb{Y}
V \in Y=\dot W^{2, 1} \cap \K_{\log} \cap L^{3/2, 1}.
\ee

Also let $X_0$ and $Y_0$ be the closures of the space of test functions $\mc D = C^\infty_c$ in $X$ and $Y$ respectively. There are differences (the local and distal properties, as well as local integrability) between the spaces $\K$, $\K_{\log}$, $\K_2$, and $\K_{2, \log^2}$ and the closure of $\mc D$ within them, see \cite{BecGol4}.

The following is our main result.

\begin{theorem}\lb{thm1} Consider a self-adjoint Hamiltonian
$$
H=-\Delta+U=-\Delta+\nabla A+V
$$
such that $A \in X_0$ and $V \in Y_0$ and assume that $0$ is neither an eigenvalue nor a resonance for $H$ or for $H_{-1} = -\Delta-U$; or alternatively let $A \in X$ and $V \in Y$ be small in norm. Then
$$
\|e^{itH} P_{ac} f\|_{L^\infty} \les |t|^{-3/2} \|f\|_{L^1}.
$$
\end{theorem}


For more estimates along the same lines, proved in the case of a scalar potential, see \cite{BecGol2}, \cite{BecGol3}, \cite{BecGol4}, and \cite{BecChe}. Such estimates, including reversed Strichartz inequalities and estimates for the wave, Klein--Gordon, and other equations, can now be proved for equations with a short-range magnetic potential, using the methods employed in this paper.

We require four derivatives for $A$, but the decay condition $A \in \mc A$ is not actually needed. A more careful proof should only require $V \in \K_{\log}$, but in the interest of brevity we state condition (\ref{Y}) as above. All our size conditions are of $L^p$ type and none involves pointwise decay; nonetheless, they correspond roughly to inverse square decay, meaning we are in the short-range case (the potential decays at infinity strictly faster than $|x|^{-1}$).

\subsection{The spectrum} The spectral measure of a self-adjoint operator, such as $H$, is supported on $\R$ and may have absolutely continuous, singular continuous, and point parts: $\sigma(H)=\sigma_{ac}(H) \cup \sigma_{sc}(H) \cup \sigma_p(H)$. To the three parts of the spectrum correspond orthogonal projections $P_{ac}$, $P_{sc}$, and $P_p$.

The results enable a comparison between $H$ and the free Laplacian $H_0=-\Delta$. In the free case, $\sigma(H_0)=\sigma_{ac}(H_0)=[0, \infty)$ and there is no singular continuous or point spectrum.

The absolutely continuous spectrum is $\sigma_{ac}=[0, \infty)$, same as in the free case $H_0=-\Delta$. The dispersive estimates we prove only hold for this part of the spectrum, which is the only one present in the free case.

In the perturbed case, under quite general assumptions there is no singular continuous spectrum. The point spectrum corresponds to the eigenvalues of $H$, which in general may be present.

The effect of the Schr\"{o}dinger or wave propagators on $P_{ac} L^2$ is dispersion and decay, just like in the free case, which is different from their effect on the bound states $P_p L^2$, which is oscillation without decay.

We call $\lambda$ an \emph{eigenvalue} of $H$ if these exists a corresponding nonzero eigenstate $f \in L^2$ such that $Hf=\lambda f$. The presence of eigenvalues in the continuous spectrum can destroy dispersive estimates. However, \cite{KocTat} proved that there are no eigenvalues inside the continuous spectrum, meaning in $(0, \infty)$.

This leaves open the possibility that $0$ is an eigenvalue, which can happen non-generically (see \cite{BecGol4} and the appendix to this paper) even if the potential is smooth and compactly supported. The threshold of the spectrum $0$ could also be a \emph{resonance}, meaning that the equation $Hf=0$ has nonzero solutions (also called resonances) $f \in L^{3, \infty} \setminus L^2$. Either of these possibilities could entail some interactions with the continuous spectrum and a different rate of decay of solutions, even after projecting away the eigenstates or resonances. Eigenstates and resonances at zero energy make a separate case-by-case analysis necessary, as carried out in \cite{ErdSch}, \cite{Yaj}, or \cite{Bec2}.

To simplify the analysis, in this paper we assume that $0$ is neither an eigenvalue nor a resonance, but a regular point of the spectrum in the meaning of Definition \ref{regular}:
\begin{assumption}\lb{a1} $0$ is neither an eigenvalue nor a resonance for the Hamiltonian $H$.
\end{assumption}

Note that Assumption \ref{a1} is required for both $H$ and $H_{-1}=-\Delta-U$. More investigations are necessary to determine whether the second condition is indeed necessary or superfluous.

\subsection{History of the problem}

In the presence of only a scalar potential ($A=0$), decay estimates for the Schr\"{o}dinger equation were proved in \cite{JoSoSo}, while for the wave equation (albeit requiring a Schwartz class potential $V$) they were proved in \cite{Bea}. This was followed by various improvements, including \cite{BecGol1}.

Several known dispersive estimates for equation (\ref{sch}) and the wave equation use the limiting absorption principle in the proof --- the fact that along the absolutely continuous spectrum the resolvent takes limiting values that are not bounded operators on $L^2$, but still belong to $B(X, X^*)$ for some Banach space $X$ (usually a weighted $L^2$ space).

The limiting absorption principle works together with the method of $H_0$-smooth perturbations, introduced in \cite{Kat}, to prove self-adjointness and $L^2$ in time smoothing estimates. This was later extended in \cite{ErGoSc1} to Strichartz estimates.

Strichartz and smoothing estimates are known for Schr\"{o}dinger and wave equations with magnetic potentials, due to works such as \cite{ErGoSc1}, \cite{ErGoSc2}, or \cite{DFVV} for Schr\"{o}dinger's equation and \cite{DanFan2} and \cite{Dan} for the wave equation. In \cite{DanFan1}, $t^{-1}$ decay estimates are proved for the wave equation with a small magnetic potential, assuming that the initial data are in weighted $H^s$ spaces.

In \cite{ErGoSc2} the authors assert that ``$L^1 \to L^\infty$ dispersive bounds are currently unknown for any $A \ne 0$'', meaning in the presence of a nontrivial magnetic potential. The situation is still the same, thus the current results are completely new.

Our proof uses the method introduced in \cite{Bec1} and \cite{BecGol1}. This involves estimates (\ref{R}) and (\ref{RR}), which are stronger than the limiting absorption principle.

\section{Preliminaries}
\subsection{Definitions and notations}
In addition to the usual Lebesgue spaces we use the Orlicz space
\be\lb{llogl}
L \log L = \{f \mid \int_{\R^3} |f|\langle \log |f|\rangle < \infty\}.
\ee

Let
\be\lb{kk}
\KK =  \{f \mid \sup_{y \in \R^3} \int_{\R^3} \frac {|f(x)| \dd x}{|x-y|^2} < \infty\},\ \|f\|_{\KK} = \sup_{y \in \R^3} \int_{\R^3} \frac {|f(x)| \dd x}{|x-y|^2}.
\ee
We also use the following logarithmic modifications of $\K$ and $\KK$:
\be\lb{klog}
\K_{\log} = \{f \mid \sup_{y \in \R^3} \int_{\R^3} \frac {|f(x)| \langle \log |x-y| \rangle\dd x}{|x-y|} < \infty\},\ \|f\|_{\K_{\log}} = \sup_{y \in \R^3} \int_{\R^3} \frac {|f(x)| \langle \log |x-y| \rangle \dd x}{|x-y|},
\ee
\be\lb{kklog}
\K_{2, \log} =  \{f \mid \sup_{y \in \R^3} \int_{\R^3} \frac {|f(x)| \langle \log |x-y| \rangle\dd x}{|x-y|^2} < \infty\},\ \|f\|_{\K_{2, \log}} = \sup_{y \in \R^3} \int_{\R^3} \frac {|f(x)| \langle \log |x-y| \rangle \dd x}{|x-y|^2},
\ee
and
\be\lb{kklog2}
\K_{2, \log^2} =  \{f \mid \sup_{y \in \R^3} \int_{\R^3} \frac {|f(x)| \langle \log |x-y| \rangle^2\dd x}{|x-y|^2} < \infty\},\ \|f\|_{\K_{2, \log}} = \sup_{y \in \R^3} \int_{\R^3} \frac {|f(x)| \langle \log |x-y| \rangle^2 \dd x}{|x-y|^2},
\ee

We also use the dual of the Kato space $\K$, denoted by $\K^*$. The dual space $\K^*$ is spanned by functions of the form
\be\lb{kstar}
\frac {f(x)} {|x-x_0|},
\ee
for $x_0 \in \R^3$ and $f \in L^\infty$, see \cite{BecGol4}. Then
$$
\bigg\|\frac {f(x)} {|x-x_0|}\bigg\|_{\K^*} \leq \|f\|_\infty.
$$
Importantly, $\K^*$ is a Banach lattice (if $|f| \leq g$ and $g \in \K^*$, then $f \in \K^*$), same as $\K$ itself and all the spaces described here.

Some estimates below also require the following logarithmic modification of $\K^*$, denoted $\K^*_{\log}$: the space spanned by functions of the form
\be\lb{klogstar}
\frac {f(x) \langle \log |x-x_0| \rangle} {|x-x_0|},
\ee
where $x_0 \in \R^3$ and $f \in L^\infty$.

The product of a function in $\K_{\log}$ and a function in $\K^*_{\log}$ is in $L^1$. It is irrelevant whether they are dual to each other (but they are). Note that $\K_{\log} \subset \K$ and $\K^* \subset \K_{\log}^*$.

We also use the following notations:

* $W^{n, p}$ denote Sobolev spaces, while $\dot W^{n, p}$ denote homogeneous Sobolev spaces. $W^{n, L \log L}$ are Orlicz--Sobolev spaces. $H^s=W^{s, 2}$.

* $a \les b$ means that there exists some constant $C$ such that $|a| \leq C |b|$.

* $\langle \cdot \rangle = \sqrt {1+(\cdot)^2}$.

* $\B(X, Y)$ is the space of bounded linear operators from $X$ to $Y$.

The Fourier transform of a function $f$ on $\R$ is
$$
\widehat f(\xi) = [\mc F f](\xi) = \frac 1 {2\pi} \int_\R f(x) e^{-ix\xi} \dd x.
$$
\subsection{The free and perturbed resolvent}
Since $\div (Af) = (\div A) f + A \cdot \nabla f$, the operators $\nabla A$ and $A \nabla$ are equivalent modulo an extra scalar potential term of $\div A$. We can rewrite $H=-\Delta+i(A\nabla+\nabla A)+V$ in the form $H=-\Delta+i\nabla \tilde A+\tilde V$, where $\tilde A = 2A$ and $\tilde V = V - \nabla A$ or $H=-\Delta+i\tilde A \nabla+V^\#$, where $\tilde A = 2A$ and $V^\# = V + \nabla A$. Either way $H=-\Delta+U$.

The resolvent of $H$ is $R=(H-\lambda)^{-1}$, while the free resolvent is $R_0=(-\Delta-\lambda)^{-1}$. They are connected by the resolvent identity:
\be\lb{res_id}
R=R_0(I+UR_0)^{-1}=(I+R_0U)^{-1}R_0.
\ee

The operator $I+R_0U$ is called the Kato--Birman operator and will play an important part in the subsequent analysis.

The spectrum of the free Laplacian is $\sigma(-\Delta)=[0, \infty)$. In three space dimensions, the free resolvent has the integral kernel
$$
R_0(\lambda^2)(x, y) = \frac 1 {4\pi} \frac {e^{i\lambda|x-y|}}{|x-y|}.
$$
The resolvent lives on the Riemann surface of the square root. In particular, for $\lambda \in \R$, $\lambda^2 \geq 0$, the resolvent takes two limiting values along the spectrum $\sigma(-\Delta)=[0, \infty)$, namely $R_0(\lambda^2 \pm i0)$.

The limiting values are not bounded operators on $L^2$, but belong to $\B(L^{6/5, 2}, L^{6, 2})$ instead. This is an instance of the limiting absorption principle. The difference of the two values $R_0(\lambda^2 \pm i0)$ is the density of the absolutely continuous part of the spectral measure, by Stone's formula \cite{ReeSim}:
$$
\dd E_{\lambda} = \frac 1 {2i} (R(\lambda^2-i0) - R(\lambda^2+i0)) \dd \lambda.
$$

The free resolvent commutes with the Laplacian and with derivatives. The integral kernel of $R_0 \nabla = \nabla R_0$ is
\be\lb{derivative}
R_0(\lambda^2)\nabla(x, y) = \frac 1 {4\pi} \bigg(\frac {i\lambda e^{i\lambda|x-y|}}{|x-y|} - \frac {e^{i\lambda|x-y|}}{|x-y|^2} \bigg) \frac {x-y}{|x-y|}.
\ee

Also note that
$$
\partial_\lambda R_0(\lambda^2)(x, y) = \frac i {4\pi} e^{i\lambda|x-y|},\ \nabla \partial_\lambda R_0(\lambda^2)(x, y) = -\frac 1 {4\pi} \lambda e^{i\lambda|x-y|} \frac {x-y}{|x-y|}.
$$


\begin{definition}\lb{regular} A point $\lambda$ in the support of $\sigma_{ac}(H)$ is a \emph{regular point} of the spectrum if $I + R_0(\lambda \pm i0) V$ is invertible.
\end{definition}
Under our conditions, all points in the absolutely continuous spectrum are regular, except for $\lambda=0$. Assuming that $0$ is neither an eigenvalue nor a resonance means $0$ is also regular.

The point part of the spectral measure is made of orthogonal projections on the eigenstates of $H$. The paper \cite{KocTat}, generalizing \cite{IonJer}, has shown that there are no eigenvalues embedded in the continuous spectrum, meaning on $(0, \infty)$.

In the appendix we prove that $H$ has finitely many negative eigenvalues.


\subsection{Algebraic structures}

\begin{definition} For any two Banach lattices $X$ and $Y$, consider the algebraic structure $\U(X, Y)$, defined as follows:
$$
\U(X, Y) := \{T(\rho, y, x) \mid \int_{\R} |T(\rho, y, x)| \dd \rho \in \B(X, Y)\}.
$$
\end{definition}
Define the composition of two operators $S \in \U(X, Y)$ and $T \in \U(Y, Z)$ by
$$
[T \circ S](\rho, z, x) = \int_{\rho_1+\rho_2=\rho} \int_{\R^3} T(\rho_1, z, y) S(\rho_2, y, x) \dd y \dd \rho_1.
$$
This structure is an algebroid, in the sense that if $S \in \U(X, Y)$ and $T \in \U(Y, Z)$ then $T \circ S \in \U(X, Z)$ and
$$
\|T \circ S\|_{\U(X, Z)} \leq \|T\|_{\U(Y, Z)} \|S\|_{\U(X, Y)}.
$$
In particular, $\U(X):=\U(X, X)$ is an algebra.

For $T \in \U(X, Y)$, its Fourier transform $\widehat T \in \widehat \U(X, Y)$ is
$$
\widehat T(\lambda, y, x) = \int_\R T(\rho, y, x) \dd \rho.
$$
Pointwise this defines a bounded operator $\widehat T(\lambda) \in \B(X, Y)$, but $\widehat \U(X, Y)$ is a proper subset of $L^\infty_\rho \B(X, Y)$. Still, the spaces $\widehat \U(X, Y)$ also form an algebroid under pointwise composition $[\widehat T  \widehat S](\lambda) = \widehat T(\lambda) \widehat S(\lambda)$.

We now relate these structures to our problem. The simplest example is the free resolvent $R_0$:
$$
\widehat {R_0(\lambda^2)}(\rho)(x, y) = \frac 1 {4\pi \rho} \delta_{|x-y|=\rho} = \frac {\sin(\rho\sqrt{-\Delta})}{\sqrt{-\Delta}}.
$$
Then
$$
\int_\R |\widehat R_0(\rho)(x, y)| \dd \rho = \frac 1 {4\pi|x-y|} = R_0(0)(x, y) = (-\Delta)^{-1}(x, y).
$$
Therefore $R_0(\lambda^2) \in \widehat U(\K, L^\infty) \cap \widehat U(L^1, \K^*)$. Here $\K^*$ is the dual of $\K$, see \cite{BecGol4} and below for its characterization.

Consequently, if $V \in \K$, $R_0(\lambda^2) V \in \widehat U(L^\infty) \cap \widehat U(\K^*)$ and $V R_0(\lambda^2) \in \widehat U(L^1) \cap \widehat U(\K)$. Note this is not true for $R_0 \nabla A$ or $A \nabla R_0$, which include a first derivative (a distribution that is not a measure).

Likewise,
$$
\widehat {\partial_\lambda R_0(\lambda^2)}(\rho)(x, y) = \frac i {4\pi} \delta_{|x-y|=\rho},
$$
so
$$
\int_\R |\widehat {\partial_\lambda R_0(\lambda^2)}(\rho)(x, y)| \dd \rho = \frac 1 {4\pi}.
$$
Consequently $\partial_\lambda R_0(\lambda^2) \in \widehat U(L^1, L^\infty)$.

For the perturbed resolvent, we start with a bilinear estimate, Proposition \ref{main_bilinear}, though in our use $U=U^\#$. Let
\be\lb{deft}
T(\lambda)(x, z) = [R_0(\lambda^2) U R_0(\lambda^2) U^\#](x, z).
\ee
Also consider the main term
\be\lb{tt}
T_1(\lambda)(x, z) = [R_0(\lambda^2) \nabla A R_0(\lambda^2) \nabla A^\#](x, z)
\ee
and
\be\lb{ttilde}
T_2(\lambda)(x, z) = \tilde T V^\#,\ \tilde T = [R_0(\lambda^2) \nabla A R_0(\lambda^2)](x, z)
\ee
(recall $\nabla A$ here means the composition of operators), as well as
$$
T_3(\lambda)(x, z) = [R_0(\lambda^2) V R_0(\lambda^2) \nabla A^\#](x, z), \text{ and } T_4(\lambda)(x, z) = [R_0(\lambda^2) V R_0(\lambda^2) V^\#](x, z),
$$
such that
\be\lb{t14}
T = T_1 + T_2 + T_3+T_4.
\ee
$\tilde T$ corresponds to the first term in the expansion (\ref{series}) and is less singular than the main term $T$.

One has
\be\lb{series}
R=(I-R_0UR_0U)^{-1}(R_0-R_0UR_0)=(I-T)^{-1}(R_0-\tilde T-R_0VR_0).
\ee
Because of the magnetic potential term $\nabla A$, it is easier to bound $R_0UR_0U$ than $R_0U$. If $T$ is sufficiently small in norm, then $(I-T)^{-1}$ is always invertible by means of a power series; otherwise we need to use Wiener's Theorem, see \cite{BecGol3}, to establish invertibility.

\begin{theorem}[Wiener]\lb{wiener} Suppose $T \in \U(X)$ is such that
	
$(*)$ $\lim_{\delta \to 0} \|T(\rho-\delta)-T(\rho)\|_{\U(X)} = 0$

$(**)$ $\lim_{R \to \infty} \|\chi_{|\rho| \geq R} T(\rho)\|_{\U(X)} = 0$.

If $I+\widehat T(\lambda)$ is invertible in $\B(X)$ for all $\lambda \in \R$, then $(I+T)^{-1}=I+S$ with $S \in \U(X)$.
\end{theorem}

As the main intermediate step, we prove that the Fourier transform of $T_1$, $\widehat T_1(\rho)(x, z)$, belongs to $\U(L^\infty)\cap \U(\K^*_{\log})$, or otherwise stated $T_1 \in \widehat \U(L^\infty) \cap \widehat U(\K^*_{\log})$, see Proposition \ref{main_bilinear}. We next prove that $(I-R_0UR_0U)^{-1} \in \widehat \U(L^\infty) \cap \widehat \U(\K^*_{\log})$, eventually arriving at a characterization of the perturbed resolvent (Proposition \ref{prop_wiener}), namely
\be\lb{R}
R(\lambda^2) \in \widehat \U(\K, L^\infty) \cap \widehat U(L^1, \K^*_{\log})
\ee
and
\be\lb{RR}
\partial_\lambda R(\lambda^2) \in \widehat U(L^1, L^\infty),
\ee
almost the same as in the free case. This implies Theorem \ref{thm1}.

Before that, we introduce elliptical and cylindrical coordinates and state a series of lemmas.

\subsection{Elliptical and cylindrical coordinates}\lb{notations}
To describe the integrals below we introduce the following notations: for two fixed points $x$ and $z$ and $y \in \R^3$, let $r_1=|x-y|$, $\vec r_1=x-y$, $r_2=|y-z|$, $\vec r_2=y-z$, $\rho=r_1+r_2$, $r=|x-z|$.

Also let
$$
\Sigma_\rho=\{y\in\R^3 \mid \rho = r_1+r_2 = |x-y|+|y-z|\}.
$$
The surface $\Sigma_\rho$ is a rotation ellipsoid with the two shortest axes equal, nonempty and non-degenerate for $\rho>r=|x-z|$.

After fixing $x$, $z$, and $\rho$, choose a coordinate system in which
$$
x=(x_1, x_2, x_3)=(-r/2, 0, 0),\ y=(y_1, y_2, y_3),\ z=(z_1, z_2, z_3)=(r/2, 0, 0).
$$
Also let
$$
o=(0, 0, 0)
$$
be the center of the ellipsoid (the midpoint of the $xz$ segment) and let $A=(A_1, A_2, A_3)$.

The two distinct semiaxes of the ellipsoid $\Sigma_\rho$ have lengths
$$
b=\frac {\sqrt{\rho^2-r^2}} e2<a=\frac \rho 2.
$$
The equation of the ellipsoid is
$$
\frac{y_1^2}{a^2}+\frac{y_2^2+y_3^2}{b^2}=1 \text{ or } \frac {y_1^2}{\rho^2} + \frac {y_2^2+y_3^2}{\rho^2 - r^2} = \frac 1 4.
$$
We parametrize $\Sigma_\rho$ by means of elliptical coordinates
$$
y_1=a \cos \theta,\ y_2=b \sin \theta \cos \phi,\ y_3=b \sin \theta \sin \phi,
$$
where $\theta \in (0, \pi)$, $\phi \in [0, 2\pi)$. Then $\ds |y|=\frac 1 2 \sqrt {\rho^2 - r^2 \sin ^2 \theta}$. Note that $\theta$ is not $\widehat{xoy}$, which is rather given by
$$
\tan \widehat {xoy} = \frac b a \tan \theta.
$$ 
In these coordinates
$$\begin{aligned}
	\vec r_1 &= (-\frac r 2 - a \cos \theta, -b \sin \theta \cos \phi, -b \sin \theta \sin \phi),\ r_1 = \frac {\rho+r \cos \theta} 2,\\
	\vec r_2 &= (-\frac r 2 + a \cos \theta, b \sin \theta \cos \phi, b \sin \theta \sin \phi),\ r_2 = \frac {\rho-r \cos \theta} 2.
\end{aligned}$$
Thus $r_1-r_2=r\cos\theta$ is independent of $\rho$, meaning that for fixed $\theta$ and $\phi$ we obtain a hyperbolic arc by varying $\rho$; for fixed $\theta$ we obtain one sheet of a two-sheeted revolution hyperboloid. Furthermore,
$$
\frac {\partial r_1}{\partial \rho} = \frac {\partial r_2}{\partial \rho} = \frac 1 2,\ \frac {d \vec r_1}{d\rho} = \frac 1 2 (-\cos \theta, \frac \rho b \sin \theta \cos \phi, \frac \rho b \sin \theta \sin \phi),\ \frac {d \vec r_2}{d\rho} = \frac 1 2 (\cos \theta, \frac \rho b \sin \theta \cos \phi, \frac \rho b \sin \theta \sin \phi).
$$


With a Jacobian factor of $J(y) \dd \mu =J(\rho, \theta, \phi) \dd \theta \dd \phi = (\rho^2 - r^2\cos^2\theta) \sin\theta \dd \theta \dd \phi = 4 r_1 r_2 \sin \theta \dd \theta \dd \phi$, for any integrable function $f$
$$
\int_{\R^3} f = \int_0^\infty \bigg(\int_{\Sigma_\rho} f(y) J(y) \dd \mu\bigg) \dd \rho = \int_0^\infty \bigg(\int_0^\pi \int_{-\pi}^\pi f(\rho, \theta, \phi) J(\rho, \theta, \phi) \dd \phi \dd \theta\bigg) \dd \rho
$$
and the differentiation formula for such an integral is
\be\lb{dif}
\frac d {d\rho} \int_{\Sigma_\rho} f(y) J(y) \dd \mu = \int_{\Sigma_\rho} \bigg(\partial_\rho f(y) + \frac {2\rho}{\rho^2-r^2 \cos^2 \theta} f(y)\bigg) J(y) \dd \mu.
\ee
Note that
$$
\frac {2\rho} {\rho^2-r^2 \cos^2 \theta} = \frac {r_1+r_2}{2r_1 r_2} = \frac 1 {2r_1} + \frac 1 {2r_2}.
$$

For each $\phi$, let $\partial_R = (\cos \phi) \partial_2 + (\sin \phi) \partial_3$ be the derivative in the radial direction, where
$$
R = \sqrt{y_2^2+y_3^2}=\frac 1 2\sqrt{\rho^2-r^2}\sin \theta,
$$
and let $y_0$ be the projection of $y$ on the $xz$ line; also, let $\int_{y_0}^y$ denote the integral on the line segment $y_0y$.

\subsection{Auxiliary results}
\begin{lemma}\lb{lema1} For any function $f \in \K_2$
	\be\lb{ineg}
	\int_{\rho \leq 2r} \frac {|f(y)| \dd y}{r \sqrt{\rho^2-r^2} \sin \theta} \les \|f\|_\KK.
	\ee
\end{lemma}
\begin{proof}
	For $\rho \leq 2r$, we integrate the inequality
	$$
	\int_{\rho \leq 2r} \frac {|f(y)| \dd y}{|y-y_0|^2} \leq \|f\|_\KK
	$$
	over $y_0$ on the line segment $\{y_0=(y_{01}, 0, 0) \mid y_0 \in [-10r, 10r]\}$. Letting $R=\sqrt {y_2^2+y_3^2}=\frac 1 2 \sqrt{\rho^2-r^2} \sin \theta$, the integral is an arctangent:
	$$
	\frac 1 R \arctan \frac a R \bigg|_{a=y_1-10r}^{a=y_1+10r} > \frac 1R
	$$
	if $\rho \leq 2r$. The conclusion follows.
\end{proof}

\begin{lemma}\lb{lema2} For any function $f \in \dot W^{1, 1}$
	$$
	\int_{\R^3} \frac {|f(y)|} R \les \|\nabla f\|_{L^1}.
	$$
\end{lemma}
\begin{proof}
	We use cylindrical coordinates $R$, $\phi$, and $y_1$. The Jacobian factor for this coordinate change is $R$.
	
	If $f \in \dot W^{1, 1}$ then $|f| \in \dot W^{1, 1}$ as well and its norm is no larger, because $|\nabla |f|| = |\nabla f|$. In cylindrical coordinates
	\be\lb{cyl}
	\int_{\R^3} \frac {|f(y)|} R = \int_\R \int_0^{2\pi} \int_0^\infty |f(y)| \dd R \dd \phi \dd y_1.
	\ee
	Write $y=(R, \phi, y_1)$ and $|f(R, \phi, y_1)| = - \int_R^\infty \partial_R |f(\tilde R, \phi, y_1)| \dd \tilde R$ and then integrate:
	$$
	(\ref{cyl}) \leq \int_\R \int_0^{2\pi} \int_0^\infty R |\partial_R f| \dd R \dd \phi \dd y_1 \leq \|\nabla f\|_{L^1}.
	$$
\end{proof}

A similar result is still true with a logarithmic correction. To study this, we use dual Orlicz spaces defined on a domain $D$ in $\R^2$, namely the exponential class $L^{exp}$
$$
L^{\exp} = \{f \mid \exists c>0 \int_D e^{c|f|} < \infty\}
$$
and $L \log L$
$$
L \log L = \{f \mid \int_D |f|(1+\log|f|) < \infty\}.
$$
With our definition $L^1 \subset L \log L$. The norms are
$$
\|f\|_{L^{\exp}(D)} = \inf \{c>0 \mid \int_D e^{|f|/c}-1 \leq 1\}
$$
and
$$
\|f\|_{L \log L(D)} = \|f\|_{L^1} + \inf \{c>0 \mid \int_D (|f|/c) \log(|f|/c) - |f|/c \leq 1 \}.
$$

These dual spaces satisfy H\"{o}lder's inequality
$$
\int_D |fg| \les \|f\|_{L^{\exp}(D)} \|g\|_{L \log L(D)}.
$$

Note that $\log R$ is in the exponential class on the disk, with norm of size $r_0^2$, where $r_0$ is the disk's radius. To avoid complications, we use disks of radius at most $1$.

\begin{lemma}\lb{lema2log} For any function $f \in \dot W^{1, L \log L}$
	$$
	\int_{D(0, R_0)} \frac {|f(y)| |\log R|} R \les \|\nabla f\|_{L \log L} + \langle \log R_0 \rangle \|\nabla f\|_{L^1}.
	$$
	Here $D(0, R_0) \subset \R^2$ is a disk of radius $R_0$.
\end{lemma}
\begin{proof}
	Without loss of generality, we can take $R_0 \leq 1$, since in the region outside this radius the statement reduces to the previous Lemma \ref{lema2}.
	
	We use a smooth cutoff function $\chi$. In cylindrical coordinates
	$$
	\int_{D(0, 2R_0)} \frac {|\log R| |f(y)| \chi(|y|/r_0)} R = -\int_0^{2\pi} \int_0^{2R_0} \log R\, |f(y)| \chi(|y|/R_0) \dd R \dd \phi.
	$$
	Integrating by parts,
	$$
	\int_0^{2R_0} \log R\, |f(y)| \chi(|y|/R_0) \dd R \les \int_0^{r_0} (\log R - 1) (\partial_R |f(y)| + |f(y)| \frac {|\chi'|}{R_0}) R \dd R.
	$$
	In the first term, $\log R$ is in the exponential class, of bounded norm, so it suffices to take $\nabla f$ in $L \log L$, while the second term is bounded by $(1+|\log R_0|) \int_{\R^2} \frac f R$ and, using Lemma \ref{lema2}, by $(1+|\log R_0|)\|\nabla f\|_{L^1}$.
\end{proof}

\begin{lemma}\lb{lema3} For any function $f$ for which the right-hand side is finite
$$
\int_{\R^3} \frac {|f(y)|} {r_1 R} \les \int_{\R^3} \frac {|\nabla f(y)|} R.
$$
\end{lemma}
\begin{proof}
	If $f \in \dot W^{1, 1}$ then $|f| \in \dot W^{1, 1}$ as well and its norm is no larger, because $|\nabla |f|| = |\nabla f|$. We use spherical coordinates centered at $x$: $r_1$, $\theta_1, \phi_1$, such that $z$ corresponds to $\theta_1=0$. The Jacobian factor is $r_1^2 \sin \theta_1 = r_1 R$. Then
	\be\lb{sph}
	\int_{\R^3} \frac {|f(y)|} {r_1 R} = \int_0^{2\pi} \int_0^\pi \int_0^\infty |f(y)| \dd r_1 \dd \theta_1 \dd \phi_1.
	\ee
	Write $|f(r_1, \theta_1, \phi_1)| = - \int_R^\infty \partial_{r_1} |f(\tilde r_1, \theta_1, \phi_1)| \dd \tilde r_1$ and then integrate:
	$$
	(\ref{sph}) \leq \int_0^{2\pi} \int_0^\pi \int_0^\infty r_1 \partial_{r_1} |f(y)| \dd r_1 \dd \theta_1 \dd \phi_1 \leq \int_{\R^3} \frac {|\nabla f|} R.
	$$
\end{proof}

Again, a logarithmically modified version of this is also true.
\begin{lemma}\lb{lema3log}
$$
\int_{D \times [a, b]} \frac {|f(y)||\log R|}{r_1 R} \les \int_{D \times [a, b]} \frac {|\nabla f(y)||\log R|}{R} + \langle \log R_0 \rangle \int_{D \times [a, b]} \frac {|\nabla f(y)|}{R} ,
$$
where $D \times [a, b]$ is a cylinder of radius $R_0$, such that the origin is on the axis of the cylinder.
\end{lemma}
\begin{proof}
	With no loss of generality, we can assume that $R_0 \leq 1$, as in the region outside this radius the estimate reduces to the previous Lemma \ref{lema3}.
	
	We use a cutoff function $\chi$. In spherical coordinates, $R=r_1 \sin \theta_1$, $R_0 = r_0 \sin \theta_1$, where $r_0=r_0(\theta_1)$ is the distance at which the half-line intersects the cylinder, and the Jacobian is $r_1 R$. The integral becomes
$$
\int_{\R^3} \frac {|f(y)||\log R|} {r_1 R} = -\int_0^{2\pi} \int_0^\pi \int_0^{r_0} |f(y)| \log R \dd r_1 \dd \theta_1 \dd \phi_1.
$$
Since $R=r_1 \sin \theta_1$, $\log R = \log r_1 + \log \sin \theta_1$. We integrate by parts:
$$
\int_0^{r_0} |f(y)| \chi(|y|/r_0)(\log r_1 + \log \sin \theta_1) \dd r_1 \les \int_0^{r_0} r_1 (\log r_1 + \log \sin \theta_1 - 1) (\partial_{r_1} |f(y)| + |f(y)| \frac {|\chi'|}{r_0}) \dd r_1.
$$
The second term in the second parenthesis is bounded by $\log R_0$ times $\ds \int_{\R^3} \frac {|f(y)|}{r_1 R}$, which is handled by Lemma \ref{lema3}.

\end{proof}

\begin{lemma}
	$\|A\|_\KK \leq \|\nabla A\|_\K \leq \|\nabla^2 A\|_{L^1}$.
\end{lemma}
\begin{proof}
	Integrate by parts in spherical coordinates, also see \cite{BecGol2}.
\end{proof}

\section{Proof of main results}
\subsection{Linear and bilinear estimates}

Recall
$$
T_1(\lambda)(x, z) = [R_0(\lambda^2) \nabla A R_0(\lambda^2) \nabla A^\#](x, z).
$$

\begin{proposition}\lb{main_bilinear} Suppose that $A \in \K_{2, \log^2}$ is such that $\nabla A \in \K_{\log} \cap \K_{2, \log^2}$, $\nabla^2 A \in \K_{2, \log^2} \cap L^1$, $\nabla^3 A, \nabla^4 A \in L \log L$, and let $A^\# \in \K_{\log} \cap \K_{2, \log^2}$. Then the operator $T_1$ given by (\ref{tt}) belongs to $\widehat \U(L^\infty) \cap \widehat \U(\K^*_{\log})$ and
$$
\|T_1\| \les Bil(A, A^\#)
$$
where
\be\lb{bil}\begin{aligned}
Bil(A, A^\#) &= \|A\|_{\K_{2, \log^2}} \|A^\#\|_{\K_{2, \log^2}} + \|\nabla A\|_{\K_{2, \log^2}} \|A^\#\|_{\K_{\log}} + \|\nabla A\|_{\K_{\log}} \|A^\#\|_{\K_{2, \log^2}} + \|\nabla^2 A\|_{L^1} \|A^\#\|_{\K_{2, \log^2}} \\
&+\|\nabla^2 A\|_{\K_{2, \log^2}} \|A^\#\|_{\K_{\log}} + \|\nabla^3 A\|_{L \log L} \|A^\#\|_{\K_{\log}} + \|\nabla^4 A\|_{L \log L} \|A^\#\|_{\K_{2, \log^2}}.
\end{aligned}\ee
\end{proposition}
The conditions imply that $A$ and its derivative are continuous and bounded.

In exactly the same way one can bound
$$
T_1^\#(\lambda) = [A^\# \nabla R_0(\lambda^2) A \nabla R_0(\lambda^2)] \in \widehat \U(L^1).
$$

\begin{proof}
The integral kernel of $T_1(\lambda)$ is
\be\lb{t}\begin{aligned}
T_1(\lambda)(x, z) &= \int_{\R^3} (\nabla_x R_0(\lambda^2)(x, y) \cdot A(y)) (\nabla_y R_0(\lambda^2)(y, z) \cdot A^\#(z)) \dd y\\
&= \frac 1 {16 \pi^2} \int_{\R^3} \bigg(\frac {i\lambda e^{i\lambda|x-y|}}{|x-y|} - \frac {e^{i\lambda|x-y|}}{|x-y|^2} \bigg) \bigg(\frac {x-y}{|x-y|} \cdot A(y)\bigg) \bigg(\frac {i\lambda e^{i\lambda|y-z|}}{|y-z|} - \frac {e^{i\lambda|y-z|}}{|y-z|^2} \bigg) \bigg(\frac {y-z}{|y-z|} \cdot A^\#(z)\bigg) \dd y.
\end{aligned}\ee

After distributing the parentheses in (\ref{t}), there are four terms to consider. The easiest term to handle in (\ref{t}) is
$$
I_0 = \int_{\R^3} \frac {e^{i\lambda\rho}(\vec r_1 \cdot A(y))(\vec r_2 \cdot A^\#(z))}{r_1^3r_2^3} \dd y.
$$

Let
$$
B_0(\rho, x, z) = \int_{\Sigma_\rho} \frac {(\vec r_1 \cdot A(y))(\vec r_2 \cdot A^\#(z))} {r_1^3r_2^3} J(y) \dd \mu.
$$
Then
$$
\int_0^\infty |B_0(\rho, x, z)| \dd \rho \leq \int_{\R^3} \frac {|A(y)| |A^\#(z)|} {|x-y|^2|y-z|^2} \dd y
$$
and
\be\lb{unu}\begin{aligned}
\sup_x \int_{\R^3} \bigg(\int_0^\infty |B_0(\rho, x, z)| \dd \rho\bigg) \dd z &\leq \sup_x \int_{\R^3} \int_{\R^3} \frac {|A(y)| |A^\#(z)|} {|x-y|^2|y-z|^2} \dd y \dd z \\
&\leq \sup_x \int_{\R^3} \frac {|A(y)| \dd y} {|x-y|^2} \sup_y \int_{\R^3} \frac {|A^\#(z)| \dd z} {|y-z|^2} \leq \|A\|_\KK^2.
\end{aligned}\ee

Consequently
$$
\|B_0\|_{\U(L^\infty_z, L^\infty_x)} \leq \|B_0\|_{L^\infty_x L^1_z L^1_\rho} \leq \|A\|_\KK \|A^\#\|_\KK.
$$
At the same time, $I_0$ can be rewritten as
$$
I_0 = \int_0^\infty e^{i\lambda \rho} B_0(\rho, x, z) \dd \rho,
$$
proving that $\widehat I_0 = B_0 \in \U(L^\infty)$.

In parallel we also prove $\U(\K^*)$ bounds. However, to deal with some logarithmic factors that appear in later computations, we shall rather use $\K^*_{\log}$ instead of $\K^*$ all throughout. For most of the subsequent bounds this is not necessary, but this logarithmic modification is needed to deal with three of the most singular terms below.

The same estimates for the integral kernel can be used to prove both kinds of bounds.

For $I_0$, we start from (\ref{unu}) and first prove that the integral kernel
\be\lb{k2bd}
K(x, y)=\frac {f(y)}{|x-y|^2} \in \B(L^\infty),\ f \in \KK,
\ee
is in $\B(\K^*)$. For any $w \in \R^3$ we test the integral kernel against $\ds\frac 1 {|y-w|}$, which are representative elements of $\K^*_y$:
$$
\int_{\R^3} \frac {|f(y)|}{|x-y|^2} \frac {|x-w|}{|y-w|} \dd y \leq \int_{\R^3} \frac {|f(y)|}{|x-y|^2} + \frac {|f(y)|}{|x-y||y-w|} \dd y \les \|f\|_{\KK}.
$$
Here we used the fact that
\be\lb{unuprim}
\sup_{x, w} \int_{\R^3} \frac {|f(y)|}{|x-y||y-w|} \leq 2\|f\|_\KK.
\ee
To prove this, we divide $\R^3$ into two regions according to whether $|x-y|<|y-w|$ or not.

Thus the output is $\ds \frac {\tilde f(x)}{|x-w|}$, where $\tilde f \in L^\infty$ is the integral above and $\|\tilde f\|_{L^\infty} \les \|f\|_\KK$, so the operator norm is bounded by $\|f\|_\KK$.

Now (\ref{unu}) is the composition of two such operators, so it is bounded on $\K^*$ of norm $\les \|A\|_\KK \|A^\#\|_\KK$.

Next, consider the integral kernel
\be\lb{k2logbd}
\tilde K(x, y) = \frac {\langle\log|x-y|\rangle f(y)}{|x-y|^2}.
\ee
We test it against $\ds\frac {\langle \log|y-w| \rangle}{|y-w|}$ for $w \in \R^3$:
$$
\int_{\R^3} \frac {\langle \log|x-y| \rangle |f(y)|}{|x-y|^2} \frac {|x-w| / \langle \log|x-w| \rangle}{|y-w| / \langle \log|y-w| \rangle} \dd y \les \int_{\R^3} \frac {\langle \log|x-y| \rangle |f(y)|}{|x-y|^2} + \frac {|f(y)|\langle \log |y-w| \rangle}{|x-y||y-w|} \dd y \les \|f\|_{\K_{2, \log^2}}.
$$
Here we use the fact that when $a, b, c \geq 0$ and $c \leq a+b$
$$
c/\langle \log c \rangle \les a/\langle \log a \rangle + b/\langle \log b \rangle.
$$
Thus the value of $\tilde K$ is a bounded function times $\ds\frac {\langle \log|x-w| \rangle}{|x-w|}$. Hence $\tilde K$ is a bounded operator on $\K^*_{\log}$, with operator norm of size at most $\|f\|_{\K_{2, \log^2}}$. Note that the same applies to the original integral kernel $K$.

Similarly,
$$
L(x, y) = \frac {f(y)}{|x-y|} \in \B(L^\infty) \cap \B(\K^*)
$$
when $f \in \K$, of norm $\|L\| \les \|f\|_\K$, and
\be\lb{klogbd}
\tilde L(x, y) = \frac {\langle \log|x-y| \rangle f(y)}{|x-y|} \in \B(L^\infty) \cap \B(\K^*_{\log})
\ee
for $f \in \K_{\log}$. Note that if $f \in \K_{\log}$, then $K \in \B(\K^*_{\log})$ as well.

Using these building blocks (\ref{k2bd}-\ref{klogbd}) and a few others, we shall prove that all ensuing terms are bounded in these spaces. For compatibility with future estimates, we also use $\widehat U(\K^*_{\log})$ to estimate $I_0$, resulting in a bound of $\|A\|_{\K_{2, \log^2}} \|A^\#\|_{\K_{2, \log^2}}$ on its norm.

The other three terms in (\ref{t}) are
$$
I_1 = -\int_{\R^3} \frac {i\lambda e^{i\lambda \rho}(\vec r_1 \cdot A(y))(\vec r_2 \cdot A^\#(z))}{r_1^3r_2^2} \dd y,\ I_2 = -\int_{\R^3} \frac {i\lambda e^{i\lambda \rho}(\vec r_1 \cdot A(y))(\vec r_2 \cdot A^\#(z))}{r_1^2r_2^3} \dd y,
$$
and
$$
I_3 = \int_{\R^3} \frac {(i\lambda)^2 e^{i\lambda \rho}(\vec r_1 \cdot A(y))(\vec r_2 \cdot A^\#(z))}{r_1^2r_2^2} \dd y.
$$
Let
$$
B_1(\rho, x, z) = \int_{\Sigma_\rho} b_1(y) J(y) \dd \mu = \int_{\Sigma_\rho} \frac {(\vec r_1 \cdot A(y))(\vec r_2 \cdot A^\#(z))}{r_1^3r_2^2} J(y) \dd \mu,
$$
$$
B_2(\rho, x, z) = \int_{\Sigma_\rho} b_2(y) J(y) \dd \mu = \int_{\Sigma_\rho} \frac {(\vec r_1 \cdot A(y))(\vec r_2 \cdot A^\#(z))}{r_1^2r_2^3} J(y) \dd\mu,
$$
and
$$
B_3(\rho, x, z) = \int_{\Sigma_\rho} b_3(y) J(y) \dd \mu = \int_{\Sigma_\rho} \frac {(\vec r_1 \cdot A(y))(\vec r_2 \cdot A^\#(z))}{r_1^2r_2^2} J(y) \dd\mu.
$$
Then for $k=1,2$
$$
I_k = -\int_0^\infty i\lambda e^{i\lambda \rho} B_k(\rho) \dd \rho = \int_0^\infty e^{i\lambda\rho} \partial_\rho B_k(\rho) \dd \rho
$$
and likewise
$$
I_3 = \int_0^\infty (i\lambda)^2 e^{i\lambda \rho} B_k(\rho) \dd \rho = \int_0^\infty e^{i\lambda\rho} \partial^2_\rho B_3(\rho) \dd \rho,
$$
provided that $B_1$, $B_2 \in \dot W^{1, 1}_\rho$ and $B_3 \in \dot W^{2, 1}_\rho$. In other words,
$$\begin{aligned}
\widehat {I_1} = \partial_\rho B_1,\ \widehat {I_2} = \partial_\rho B_2,\ \widehat {I_3} = \partial_\rho^2 B_3.
\end{aligned}$$

Carrying out the differentiation according to (\ref{dif}),
$$\begin{aligned}
\widehat I_1(\rho) &= \int_{\Sigma_\rho} \bigg(\partial_\rho b_1 + \bigg(\frac 1 {2r_1} + \frac 1 {2r_2}\bigg) b_1\bigg) J(y) \dd \mu,\\
\widehat I_2(\rho) &= \int_{\Sigma_\rho} \bigg(\partial_\rho b_2 + \bigg(\frac 1 {2r_1} + \frac 1 {2r_2}\bigg) b_2\bigg) J(y) \dd \mu,\\
\widehat I_3(\rho) &= \partial_\rho \int_{\Sigma_\rho} \bigg(\partial_\rho b_3 + \bigg(\frac 1 {2r_1} + \frac 1 {2r_2}\bigg) b_3\bigg) J(y) \dd \mu \\
&= \int_{\Sigma_\rho} \bigg(\partial^2_\rho b_3 + \bigg(\frac 1 {r_1} + \frac 1 {r_2}\bigg) \partial_\rho b_3 + \frac 1 {2r_1r_2} b_3\bigg) J(y) \dd \mu.
\end{aligned}$$

Out of this list, $\frac {b_1} {2r_2}$, $\frac {b_2}{2r_1}$, and $\frac {b_3}{2r_1r_2}$ contain factors of $r_1^{-2} r_2^{-2}$, hence are integrable in the same way as (\ref{unu}). We are left with
$$
\int_{\Sigma_\rho} \bigg(\partial_\rho b_1 + \frac {b_1}{2r_1} + \partial_\rho b_2 + \frac {b_2}{2r_2} + \partial_\rho^2 b_3 + \bigg(\frac 1 {r_1} + \frac 1 {r_2}\bigg) \partial_\rho b_3\bigg) J(y) \dd \mu.
$$

Start from
$$
\partial^2_\rho b_3 = \partial_\rho \big(\partial_\rho (r_1^{-2} r_2^{-2}) [(\vec r_1 \cdot A(y))(\vec r_2 \cdot A^\#(z))] + (r_1^{-2} r_2^{-2}) \partial_\rho[(\vec r_1 \cdot A(y))(\vec r_2 \cdot A^\#(z))]\big).
$$
In the first term
$$
\partial_\rho (r_1^{-2} r_2^{-2}) = -2 r_1^{-3} \partial_\rho r_1 r_2^{-2} + r_1^{-2} (-2 r_2^{-3} \partial_\rho r_2) = -r_1^{-3} r_2^{-2} - r_1^{-2} r_2^{-3}.
$$
Thus, the contribution of this term exactly cancels $\partial_\rho b_1$ and $\partial_\rho b_2$. What remains is
$$\begin{aligned}
&\partial_\rho ((r_1^{-2} r_2^{-2}) \partial_\rho[(\vec r_1 \cdot A(y))(\vec r_2 \cdot A^\#(z))]) = \\
&= \partial_\rho (r_1^{-2} r_2^{-2}) \partial_\rho[(\vec r_1 \cdot A(y))(\vec r_2 \cdot A^\#(z))] + (r_1^{-2} r_2^{-2}) \partial^2_\rho [(\vec r_1 \cdot A(y))(\vec r_2 \cdot A^\#(z))].
\end{aligned}$$

Let
$$
b_0 = \partial_\rho (r_1^{-2} r_2^{-2}) \partial_\rho[(\vec r_1 \cdot A(y))(\vec r_2 \cdot A^\#(z))] = -(r_1^{-3} r_2^{-2} + r_1^{-2} r_2^{-3}) \partial_\rho[(\vec r_1 \cdot A(y))(\vec r_2 \cdot A^\#(z))].
$$
At the same time,
$$\begin{aligned}
\bigg(\frac 1 {r_1} + \frac 1 {r_2}\bigg) \partial_\rho b_3 &= (r_1^{-1} + r_2^{-1}) \big(\partial_\rho (r_1^{-2} r_2^{-2}) [(\vec r_1 \cdot A(y))(\vec r_2 \cdot A^\#(z))] + (r_1^{-2} r_2^{-2}) \partial_\rho[(\vec r_1 \cdot A(y))(\vec r_2 \cdot A^\#(z))]\big) \\
&= -(r_1^{-4} r_2^{-2} + 2 r_1^{-3} r_2^{-3} + r_1^{-2} r_2^{-4}) [(\vec r_1 \cdot A(y))(\vec r_2 \cdot A^\#(z))] \\
&+ (r_1^{-1}+r_2^{-1}) r_1^{-2} r_2^{-2} \partial_\rho[(\vec r_1 \cdot A(y))(\vec r_2 \cdot A^\#(z))].
\end{aligned}$$
The first term partly cancels $\ds\frac {b_1}{2r_1}$ and $\ds\frac {b_2}{2r_2}$, while the contribution of $2r_1^{-3}r_2^{-3}$ can be treated as in (\ref{unu}). The second term exactly cancels $b_0$.

After these cancellations we are left with
\be\lb{rest}
(r_1^{-2} r_2^{-2}) \partial^2_\rho [(\vec r_1 \cdot A(y))(\vec r_2 \cdot A^\#(z))] - \frac {b_1}{2r_1} - \frac {b_2}{2r_2}.
\ee

The main issue for $\ds \frac {b_1}{r_1}$ and $\ds \frac {b_2}{r_2}$ is that they contain factors of size $r_1^{-3}$ or $r_2^{-3}$, which are not integrable in $\R^3$. Thus, if we tried to integrate directly in $y$ as in (\ref{unu}), the right-hand side would be infinite. However, these terms are singular integral operators and exhibit cancellations. These cancellations lead to gains when integrating on the ellipsoids $\Sigma_\rho$, so for an optimal bound we first integrate on $\Sigma_\rho$ and then in $\rho$, unlike in (\ref{unu}).

Consider
\be\lb{b1r1}
\frac {b_1} {r_1} = \frac {(\vec r_1 \cdot A(y))(\vec r_2 \cdot A^\#(z))}{r_1^4 r_2^2} = \sum_{k, \ell=1}^3 \frac {(x_k-y_k) A_k(y) (y_\ell-z_\ell) A^\#_\ell(z)}{r_1^4 r_2^2}.
\ee
When $2 \leq k, \ell \leq 3$, $|x_k-y_k|$ and $|x_\ell-y_\ell|$ are both at most $\min(r_1, r_2)$, so both can be used to cancel two powers of $r_1$, leaving an expression that can be bounded as (\ref{unu}). If $k=1$ and $2 \leq \ell \leq 3$, the same considerations apply, since $|x_1-y_1| \leq r_1$ and $|y_\ell-z_\ell| \leq \min(r_1, r_2)$.

If $1 \leq k \leq 3$ and $\ell = 1$, consider separately the regions
$$
D_0 = \bigg\{y \mid r_1 > r_{10} = \frac {r} {100}\bigg\},\ D=\{y \mid r_1 \leq r_{10}\}.
$$
In the first region $D_0$, $r_2 \leq r_1+r \leq 101 r_1$, so $r_1^{-1} \les r_2^{-1}$ and $r_1^{-4} r_2^{-2} \les r_1^{-3} r_2^{-3}$, meaning we are in the same situation as in (\ref{unu}-\ref{unuprim}).

Inside $D$, $r$, $r_2$, and $y_1-z_1$ are comparable, in the sense that their differences are of size $r_1$, which is smaller by a factor of $100$: $\ds \rho \leq r + \frac r {100}$ and
$$
\bigg|\frac {y_1-z_1}{r_2^2} - \frac {1}{r_2}\bigg| \les \frac {r_1}{r_2^2}. 
$$
Thus, we can replace
$$
\frac {(x_k-y_k) A_k(y) (y_1-z_1) A_1^\#(z)}{r_1^4 r_2^2}
$$
by
\be\lb{patru}
\frac {(x_k-y_k) A_k(y) A_1^\#(z)}{r_1^4 r_2}
\ee
inside $D$, modulo a term that can be bounded as in (\ref{unu}) and the ensuing discussion.

To bound (\ref{patru}), we replace $A_k(y)$ by $A_k(x)$ within $D$. We are required to evaluate two terms: the integral after replacing $A_k$ by a constant function and the integral of the difference.

For $y \in D$, we then express the differences $A_k(y)-A_k(x)$ as integrals of $\nabla A_k(\tilde y)$ over the line segment connecting $y$ and $x$.

Using Fubini's theorem, we then first compute the integral in $y$, in spherical coordinates, leaving the integrand as a function of $\tilde y$.

Integrating the factor $\ds\bigg|\frac {x_k-y_k}{r_1^4 r_2}\bigg| \les r_1^{-3} r_2^{-1}$ for $r_1 \in [\tilde r_1, r_{10}]$, using spherical coordinates such that the Jacobian contains a factor of $r_1^2$, we get $\tilde r_1^{-2} \log (r_{10}/\tilde r_1) r_2^{-1}$ (recall $r_2$ is comparable to $r$, which is constant for the purpose of this integration). This leads to a bound of
$$
\int_{D} \frac {|\nabla A_k(\tilde y)| \log (r_{10}/\tilde r_1) |A_1^\#(z)| \dd \tilde y}{\tilde r_1^2 r_2}.
$$
Here $\log(r_{10}/\tilde r_1) \les \log_+ r_{10} + \log_- \tilde r_1$.

Integrating in $\tilde y$ and $z$ leads to a $\widehat U(L^\infty)$ bound of $\|\nabla A\|_{\K_{2, \log}} \|A^\#\|_\K + \|\nabla A\|_\KK \|A^\#\|_{\K_{\log}}$ for this difference term, as well as to a $\widehat U(\K^*_{\log})$ bound of $\|\nabla A\|_{\K_{2, \log^2}} \|A^\#\|_{\K_{\log}}$. See the discussion around (\ref{k2logbd}--\ref{klogbd}). This is one term that makes the logarithmic modification necessary.


Next, consider (\ref{patru}) after replacing $A_k(y)$ by $A_k(x)$ for $y \in D$. For $2 \leq k \leq 3$, due to symmetries, the integral is $0$.

For $k=1$ we compute the integral, first for fixed $\rho$, considering the integrand as a function of $\theta$. The integral is symmetric in $\phi$, which therefore contributes a factor of $2\pi$. Recall the Jacobian factor $J(y)$ is $4r_1 r_2 \sin \theta$, so the integrand is
$$
\frac {(x_1-y_1)A_1(x)A_1(z)}{r_1^3} \sin \theta.
$$
The interval endpoints are $\pi$ and $\theta_0$, where $t_0 = \cos \theta_0$, $\ds r_{10} = \rho + r t_0 = \frac {r}{100}$. The antiderivative is
$$
\int_{\Theta}^{\tilde \Theta} \frac {-r-\rho \cos \theta}{(\rho+r \cos \theta)^3} \sin \theta \dd \theta = -\frac {r^2+2r\rho t+\rho^2}{2r^2(\rho+rt)^2} \bigg|_{t=\cos \Theta}^{t=\cos \tilde \Theta}.
$$
At $t=-1$ corresponding to $\theta=\pi$ this antiderivative evaluates to $\ds\frac 1 {2r^2}$. At $t_0=\cos \theta_0$ we get
$$
\frac 1 {2r^2} + \frac {1-t_0^2} {2r_{10}^2}.
$$
The difference of the endpoint values is of size $r^{-2}$, leading to
$$
\bigg| \int_0^{\theta_0} \frac {(x_1-y_1)A_1(x)A_1^\#(z)}{r_1^3} \sin \theta \dd \theta \bigg| \les r^{-2} |A(x)| |A^\#(z)|.
$$
Integrating in $\rho$ over $\ds[r, r+\frac r{100}]$, which are the values for which $\Sigma_\rho$ may intersect $D$, results in an overall estimate of $r^{-1} |A_1(x)| |A^\#(z)| \les r^{-1} \|A\|_{L^\infty} |A^\#(z)|$. Integrating in $z$ as well leads to a $\widehat U(L^\infty)$ bound of $\|A\|_{L^\infty} \|A^\#\|_\K$
and to a $\|A\|_{L^\infty} \|A^\#\|_{\K_{\log}}$ bound in $\widehat U(\K^*_{\log})$.

The term that corresponds to $\ds \frac {b_2}{r_2}$ can be treated in the same manner, resulting in a bound of
$$
\int_{D} \frac {|\nabla A_1(y)| \log (r_{20}/r_2) |A_\ell^\#(z)| \dd y}{r_1 r_2^2}
$$
for the difference, meaning $\|\nabla A\|_{\K_{\log}} \|A^\#\|_\KK + \|\nabla A\|_\K \|A^\#\|_{\K_{2, \log}}$ in $\widehat U(L^\infty)$ and $\|\nabla A\|_{\K_{\log}} \|A^\#\|_{\K_{2, \log^2}}$ in $\widehat U(\K^*_{\log})$. The main term after substituting $A_k(z)$ for $A_k(y)$ satisfies the same bound as above.

The other expression in (\ref{rest}) is, omitting summation and a factor of $r_1^{-2} r_2^{-2}$,
\be\lb{leibniz}\begin{aligned}
\partial^2_\rho [(x_k-y_k)A_k(y)(y_\ell-z_\ell)A^\#_\ell(z)] &= \partial^2_\rho [(x_k-y_k)(y_\ell-z_\ell)] A_k(y) A^\#_\ell(z) \\
&+ 2 \partial_\rho [(x_k-y_k)(y_\ell-z_\ell)] (\partial_\rho A_k(y)) A^\#_\ell(z) \\
&+ [(x_k-y_k)(y_\ell-z_\ell)] (\partial^2_\rho A_k(y)) A^\#_\ell(z).
\end{aligned}\ee

Now we have to handle differentiation with respect to $\rho$. Let
$$
v=(v_1, v_2, v_3)=\partial_\rho y = \bigg(\frac {\cos \theta} 2, \frac {\rho \sin \theta \cos \phi} {2\sqrt {\rho^2-r^2}} , \frac {\rho \sin \theta \sin \phi} {2\sqrt {\rho^2-r^2}} \bigg).
$$
For any function $f$
$$
\ds |\partial_\rho f(y)| \les \bigg(\cos \theta + \frac {\rho \sin \theta} {\sqrt {\rho^2-r^2}}\bigg) |\nabla f(y)| \les \frac {\sqrt {\rho^2 - r^2 \cos^2 \theta}}{\sqrt {\rho^2-r^2}} |\nabla f(y)|.
$$
Note that
$$
\partial_\rho \frac \rho {\sqrt {\rho^2-r^2}} = - \frac {r^2} {(\rho^2-r^2)^{3/2}}.
$$
Also, $\sqrt {\rho^2 - r^2 \cos^2 \theta} = r_1^{1/2} r_2^{1/2}$ and $\rho^2 \sin^2 \theta \leq \rho^2 - r^2 \cos^2 \theta$, so
$$
\rho \sin \theta \leq r_1^{1/2} r_2^{1/2}.
$$
This bound is used when estimating almost every term below. Furthermore, $\sqrt{\rho^2-r^2} \sin \theta \les \min(r_1, r_2)$.

The factors of $\ds\frac 1 {\rho^2-r^2}$ and $\ds \frac 1 {(\rho^2-r^2)^{3/2}}$ that appear through differentiation pose a special difficulty that needs to be addressed. When $\rho>2r$, $\ds\frac \rho {\sqrt{\rho^2-r^2}}$ is of order $1$, but when $\rho \leq 2r$ this factor is unbounded (in a compact region, however). The two situations are distinguished in the analysis of each term below.

In the first term in (\ref{leibniz}), if $2 \leq k, \ell \leq 3$ or $k=\ell=1$, then $\partial^2_\rho [(x_k-y_k)(y_\ell-z_\ell)]$ is a constant that depends only on $\theta$ and $\phi$, of absolute value at most $1$. This case can be handled as in (\ref{unu}) and the subsequent discussion, leaving $k=1$ and $2 \leq \ell \leq 3$ or vice-versa.

For $k=1$ and $2 \leq \ell \leq 3$ (and symmetrically when $2 \leq k \leq 3$ and $\ell=1$)
\be\lb{l3}
\partial^2_\rho [(x_k-y_k)(y_\ell-z_\ell)] = 2 \bigg(-\frac 1 2 \cos \theta \frac {\rho \sin \theta g(\phi)}{2\sqrt{\rho^2-r^2}}\bigg) + \frac 1 2 (\rho \cos \theta + r) \frac {r^2 \sin \theta g(\phi)}{2(\rho^2-r^2)^{3/2}},
\ee
where $g(\phi)$ is $\sin \phi$ or $\cos \phi$. The two terms' contributions are handled separately. The first term in (\ref{l3}) produces a contribution of
\be\lb{1st}
\frac {\rho \sin \theta \cos \theta g(\phi) A_1(y) A^\#_\ell(z)}{r_1^2 r_2^2 \sqrt{\rho^2-r^2}}.
\ee
For $\rho>2r$
$$
(\ref{1st}) \les \frac {|A_1(y)| |A^\#_\ell(z)|}{r_1^2 r_2^2}
$$
and this can be handled as in (\ref{unu}), for a bound of $\|A\|_\KK \|A^\#\|_\KK$ in $\widehat U(L^\infty)$ and a bound of $\|A\|_{\K_{2, \log^2}} \|A^\#\|_{\K_{2, \log^2}}$ in $\widehat U(\K^*_{\log})$.

For $\rho\leq 2r$
$$
(\ref{1st}) \les \frac {|A_1(y)| |A^\#_\ell(z)|}{r_1^{3/2} r_2^{3/2} \sqrt{\rho^2-r^2}} = \frac {\sin \theta |A_1(y)| |A^\#_\ell(z)|}{r_1^{3/2} r_2^{3/2} R} \les \frac {|A_1(y)| |A^\#_\ell(z)|}{\rho r_1 r_2 R}.
$$
The $\widehat \U(L^\infty)$ bound can be obtained by the inequalities
$$
\int_{\R^3} \frac {|A_1(y)| \dd y}{r_1 R} \les \|\nabla^2 A\|_{L^1},\ \int_{\R^3} \frac {|A^\#_\ell(z)|}{r_2^2} \les \|A^\#\|_\KK,
$$
resulting in $\|\nabla^2 A\|_{L^1} \|A^\#\|_\KK$. For the $\widehat \U(\K^*_{\log})$ bound, we test the integral kernel, as for (\ref{k2logbd}), against $\ds \frac {\langle \log |z-w| \rangle}{|z-w|}$ for $w \in \R^3$:
\be\lb{long}
\frac {|A_1(y)| |A^\#_\ell(z)|}{\rho r_1 r_2 R} \frac {|x-w|/\langle \log|x-w| \rangle}{|z-w|/\langle \log|z-w| \rangle} \les \frac {|A_1(y)| |A^\#_\ell(z)|}{\rho r_1 r_2 R} + \frac {|A_1(y)| |A^\#_\ell(z)|}{\rho r_1 R} \frac {\langle \log|z-w| \rangle}{|z-w|} + \frac {|A_1(y)| |A^\#_\ell(z)|}{\rho r_2 R} \frac {\langle \log|z-w| \rangle}{|z-w|}
\ee
The first term has been bounded above. We split the second term into $\ds \int \frac {|A_1(y)|}{r_1 R}$ and $\ds\int \frac {A^\#(z) \langle \log |z-w| \rangle}{r_2 |z-w|}$, after replacing $\rho$ by $r_2$. We use the same split for the last term, after replacing $\rho$ by $r_1$. This results in an overall bound of $\|\nabla^2 A\|_{L^1} \|A^\#\|_{\K_{2, \log^2}}$ for the $\widehat \U(\K^*_{\log})$ norm.



The second term in (\ref{l3}) gives rise to
\be\lb{expr}
\frac {(\rho \cos \theta + r) r^2 \sin \theta g(\phi) A_1(y) A^\#_\ell(z)}{4r_1^2 r_2^2 (\rho^2-r^2)^{3/2}}.
\ee
We retain this expression for now and will handle it later, see (\ref{can}). Again, when $\rho > 2r$, this is bounded by
$$
\frac {|A_1(y)| |A^\#_\ell(z)|}{r_1 r_2^2 \rho},
$$
which leads to a bound of $\|A\|_\KK \|A^\#\|_\KK$ in $\widehat U(L^\infty)$ and $\|A\|_{\K_{2, \log^2}} \|A^\#\|_{\K_{2, \log^2}}$ in $\widehat U(\K^*_{\log})$.

For the second term in (\ref{leibniz}) we again consider several cases. When $2 \leq k, \ell \leq 3$,
$$
\partial_\rho [(x_k-y_k)(y_\ell-z_\ell)] = 2\rho \sin^2 \theta \, g(\phi),
$$
where $g(\phi)$ is $\sin^2\phi$, $\sin \phi \cos \phi$, or $\cos^2\phi$.
For $\rho>2r$ the integrand is bounded by
$$
\frac {\rho \sin^2 \theta |\nabla A_k(y)| |A^\#_\ell(z)|}{r_1^2 r_2^2} \les \frac {|\nabla A_k(y)| |A^\#_\ell(z)|}{r_1 r_2 \rho} \les \frac {|\nabla A_k(y)| |A^\#_\ell(z)|}{r_1 r_2^2}.
$$
After integrating in $y$ and $z$, we get a bound of $\|\nabla A\|_\K \|A^\#\|_\KK$ in $\widehat \U(L^\infty)$ and a bound of $\|\nabla A\|_{\K_{\log}} \|A^\#\|_{\K_{2, \log^2}}$ in $\widehat U(\K^*_{\log})$ for this expression.

For $\rho \leq 2r$ the same integrand is of size
$$
\frac {\rho \sin^2 \theta |\partial_\rho A_k(y)| |A^\#_\ell(z)|}{r_1^2 r_2^2} \les \frac {\rho \sin^2 \theta |\nabla A_k(y)| |A^\#_\ell(z)|}{r_1^{3/2} r_2^{3/2} \sqrt {\rho^2-r^2}} = \frac {\rho \sin^3 \theta |\nabla A_k(y)| |A^\#_\ell(z)|}{r_1^{3/2} r_2^{3/2} R} \les \frac {|\nabla A_k(y)| |A^\#_\ell(z)|}{\rho^2 R}.
$$
Using Lemma \ref{lema2}, we end up with bounds of $\|\nabla^2 A\|_{L^1} \|A^\#\|_\KK$ in $\widehat U(L^\infty)$ and $\|\nabla^2 A\|_{L^1} \|A^\#\|_{\K_{2, \log^2}}$ in $\widehat U(\K^*_{\log})$, in the same manner as in (\ref{long}).

For the second term in (\ref{leibniz}), when $k=1$ and $2 \leq \ell \leq 3$ or in the symmetric case,
\be\lb{l2}
\partial_\rho [(x_1-y_1)(y_\ell-z_\ell)] = -\frac {\cos \theta} 2 \frac {\sqrt {\rho^2-r^2} \sin \theta g(\phi)} 2 - \frac {\rho \cos \theta + r} 2 \frac {\rho \sin \theta g(\phi)}{2\sqrt {\rho^2-r^2}}.
\ee
We treat the two terms in (\ref{l2}) separately.

The contribution of the first term in (\ref{l2}) is
$$
\frac {\sqrt {\rho^2-r^2} \sin \theta \cos \theta g(\phi) |\partial_\rho A_1(y)| |A_\ell(z)|}{r_1^2 r_2^2} \les \frac {\sin \theta \cos \theta g(\phi) \sqrt {\rho^2-r^2 \cos^2 \theta} |\nabla A_1(y)| |A^\#_\ell(z)|}{r_1^2 r_2^2}.
$$
When $\rho>2r$ the integrand is bounded by
$$
\frac {\sqrt{\rho^2-r^2} \sin \theta |\nabla A_1(y)| |A^\#_\ell(z)|}{r_1^2 r_2^2} \les \frac {\min(r_1, r_2) |\nabla A_1(y)| |A^\#_\ell(z)|}{r_1^2 r_2^2} \les \frac {|\nabla A_1(y)| |A^\#_\ell(z)|}{r_1 r_2^2}.
$$
Again we get a bound of $\|\nabla A\|_\K \|A^\#\|_\KK$ in $\widehat \U(L^\infty)$ and $\|\nabla A\|_{\K_{\log}} \|A^\#\|_{\K_{2, \log^2}}$ in $\widehat \U(\K^*_{\log})$. For $\rho \leq 2r$ we instead get
$$
\frac {\sin \theta |\nabla A_1(y)| |A_\ell^\#(z)|}{r_1^{3/2} r_2^{3/2}} \les \frac {|\nabla A_1(y)| |A_\ell^\#(z)|}{r_1 r_2 \rho},
$$
leading to the same bounds.

Regarding the second term in (\ref{l2}), for $\rho>2r$ we use the fact that $|\rho \cos \theta \pm r| \les \rho$, so the integrand is of size
$$
\frac {\rho \sin \theta |\nabla A_1(y)| |A_\ell(z)|}{r_1^2 r_2^2} \les \frac {|\nabla A_1(y)| |A^\#_\ell(z)|}{r_1^{3/2} r_2^{3/2}}.
$$
This leads, for example, to a bound of $\|\nabla A\|_{\K} \|A^\#\|_{\KK} + \|\nabla A\|_\KK \|A^\#\|_\K$ in $\widehat \U(L^\infty)$ and $\|\nabla A\|_{\K_{\log}} \|A^\#\|_{\K_{2, \log^2}} + \|\nabla A\|_{\K_{2, \log^2}} \|A^\#\|_{\K_{\log}}$ in $\widehat \U(\K^*_{\log})$.

For $\rho \leq 2r$ we cannot use absolute value bounds. The integrand is
$$
-\frac {(\rho \cos \theta + r) \rho \sin \theta g(\phi) (v \cdot \nabla A_1(y)) A^\#_\ell(z)}{4r_1^2 r_2^2 \sqrt{\rho^2-r^2}}.
$$
The first component of $v$, $v_1$, is less singular and just for this part we again obtain a bound of
$$
\les \frac {\rho \sin^2 \theta |\nabla A_1(y)| |A_\ell^\#(z)|}{r_1 r_2^2 R} \les \frac {|\nabla A_1(y)| |A_\ell^\#(z)|}{\rho r_2 R}.
$$
This leads to a bound of $\|\nabla^2 A\|_{L^1} \|A^\#\|_\KK$ in $\widehat U(L^\infty)$ and one of $\|\nabla^2 A\|_{L^1} \|A^\#\|_{\K_{2, \log^2}}$ (as in (\ref{long})) in $\widehat U(\K^*_{\log})$.

We are left with
\be\lb{exp}
-\frac {(\rho \cos \theta + r) \rho^2 \sin^2 \theta g(\phi) \partial_R A_1(y) A^\#_\ell(z)}{8 r_1^2 r_2^2 (\rho^2-r^2)}.
\ee
Note that there are two copies of this expression, due to the factor of $2$ in front of it in (\ref{leibniz}). For now we use one copy and we keep one for later, see (\ref{suma}).

In one copy of (\ref{exp}) we replace $\partial_R A_1(y)$ by
$$
\frac 1 R \int_{y_0}^y \partial_R A_1(\eta) \dd \eta = \frac {2(A_1(y) - A_1(y_0))} {\sqrt{\rho^2-r^2} \sin \theta}
$$
and then evaluate the difference separately. After the replacement, integrating in $\phi$ makes the $A_1(y_0)$ term vanish, since
$$
\int_0^{2\pi} g(\phi) \dd \phi = 0,
$$
leaving an integrand of
\be\lb{int}
-\frac {(\rho \cos \theta + r) \rho^2 \sin \theta g(\phi) A_1(y) A^\#_\ell(z)}{4 r_1^2 r_2^2 (\rho^2-r^2)^{3/2}}.
\ee
Finally, (\ref{int}) can be combined with (\ref{expr}), resulting in a cancellation:
\be\lb{can}
(\ref{int})+(\ref{expr}) = -\frac {(\rho \cos \theta + r) \sin \theta g(\phi) A_1(y) A^\#_\ell(z)}{4 r_1^2 r_2^2 (\rho^2-r^2)^{1/2}}.
\ee
The ensuing expression can be estimated as follows:
$$
(\ref{can}) \les \frac {\sin^2 \theta |A_1(y)| |A^\#_\ell(z)|}{r_1 r_2^2 R} \les \frac {|A_1(y)| |A^\#_\ell(z)|}{\rho^2 r_2 R}
$$
which results in a $\U(L^\infty)$ bound of $\|\nabla^2 A\|_{L^1} \|A^\#\|_\KK$, by Lemmas \ref{lema2} and \ref{lema3}. The $\U(\K^*_{\log})$ bound, obtained as in (\ref{long}), is $\|\nabla^2 A\|_{L^1} \|A^\#\|_{\K_{2, \log^2}}$.

The difference also presents a gain:
$$
\bigg|\partial_R A_1(y) - \frac 1 R \int_{y_0}^y \partial_R A_1(\eta) \dd \eta\bigg| \leq \frac 1 R \int_{y_0}^y |\partial_R A_1(y) - \partial_R A_1(\eta)| \dd \eta \les R \|\nabla^2 A\|_{L^\infty}.
$$
This factor of $R$ cancels one power of $R$ in the denominator, making the whole expression integrable.

A better estimate with the same effect is
$$
\bigg|\partial_R A_1(y) - \frac 1 R \int_{y_0}^y \partial_R A_1(\eta) \dd \eta\bigg| = \frac 1 R \bigg|\int_{y_0}^y \partial_R A_1(y) - \partial_R A_1(\eta) \dd \eta\bigg| \leq \frac 1 R \int_{y_0}^y (R-R_\eta) |\nabla^2 A(\eta)| \dd \eta.
$$
Here $R-R_\eta=|y-\eta|$, where $R=R_y=\sqrt{y_1^2+y_2^2}$ and $R_\eta=\sqrt{\eta_1^2+\eta_2^2}$ is the corresponding value for $\eta$.

Using this formula instead in (\ref{exp}) leads to
$$\begin{aligned}
\les \frac {\rho^2\sin^4 \theta (\int_{y_0}^y (R-R_\eta)|\nabla^2A(\eta)|\dd \eta) A_\ell^\#(z)}{r_1 r_2^2 R^3} &\les \frac {r_1 (\int_{y_0}^y (R-R_\eta)|\nabla^2A(\eta)|\dd \eta) A_\ell^\#(z)}{\rho^2 R^3} \\
&\les \frac {(\int_{y_0}^y (R-R_\eta)|\nabla^2A(\eta)|\dd \eta) A_\ell^\#(z)}{\rho R^3}.
\end{aligned}$$
Using Fubini's theorem, we carry out the integration in $y$ first, in cylindrical coordinates, and leave the integral in $\eta$. There is a factor of $R$ from using cylindrical coordinates, so the integral is
$$
\int_\eta^{y_1} \frac {R_y-R_\eta}{R_y^2},
$$
where the upper bound $y_1$ is comparable to $r$ (and denotes the point where the half-line intersects $\Sigma_{\rho=2r}$).

The integral is at most $1+|\log R_\eta/r|$ and now there is a factor of $1/R_\eta$ because the computations were done in cylindrical coordinates. We are left with
$$
\frac {(1+|\log R_\eta|+|\log r|)|\nabla^2 A(\eta)| |A_\ell^\#(z)|}{\rho R_\eta}.
$$
This is bounded by $\|\nabla^3 A\|_{L \log L} \|A^\#\|_\K + \|\nabla^3 A\|_{L^1} \|A^\#\|_{\K_{\log}}$ in $\widehat \U(L^\infty)$, by Lemma \ref{lema2log}, and $\|\nabla^3 A\|_{L \log L} \|A^\#\|_{\K_{\log}}$ in $\widehat \U(\K^*_{\log})$.

This is again a term that requires the use of logarithmically modified norms.

In the second term in (\ref{leibniz}), in the final case when $k=\ell=1$ the integrand is
$$
\frac {2\rho \cos^2 \theta\, \partial_\rho A_1(y)\, A^\#_1(z)}{r_1^2 r_2^2}.
$$
For $\rho>2r$ this is at most of size
$$
\frac {\rho |\nabla A_1(y)| |A^\#_1(z)|}{r_1^2 r_2^2}
$$
and writing $\rho=r_1+r_2$ we obtain a bound of $\|\nabla A\|_\KK \|A^\#\|_\K + \|\nabla A\|_\K \|A^\#\|_\KK$ (with logarithmic modifications for the $\widehat \U(\K^*_{\log})$ norm). For $\rho \leq 2r$ the integrand is bounded by
$$
\frac {\rho |\nabla A_1(y)| |A^\#_1(z)|}{r_1^{3/2} r_2^{3/2} \sqrt {\rho^2-r^2}} \les \frac {|\nabla A_1(y)| |A^\#_1(z)|}{r_1 r_2 R}
$$
This leads to a $\widehat \U(L^\infty)$ bound of $\|\nabla^3 A\|_{L^1} \|A^\#\|_{\K}$, by means of Lemmas \ref{lema2} and \ref{lema3}. To obtain a $\widehat \U(\K^*_{log})$ bound, we test the integral kernel against $\ds\frac {\langle \log|z-w| \rangle}{|z-w|}$ for $w \in \R^3$:
$$
\frac {|\nabla A_1(y)| |A^\#_\ell(z)|}{r_1 R r_2} \frac {|x-w|/\langle \log|x-w| \rangle}{|z-w|/\langle \log|z-w| \rangle} \les \frac {|\nabla A_1(y)| |A^\#_\ell(z)|}{r_1 R r_2} + \frac {|\nabla A_1(y)| |A^\#_\ell(z)|}{r_1 R}\frac {\langle \log|z-w| \rangle}{|z-w|} + \frac {|\nabla A_1(y)| |A^\#_\ell(z)|}{R r_2}\frac {\langle \log|z-w| \rangle}{|z-w|}.
$$
This produces a bound of $\|\nabla^3 A\|_{L^1} \|A^\#\|_{\K_{\log}} + \|\nabla^2 A\|_{L^1} \|A^\#\|_{\K_{2, \log^2}}$.

The last term in (\ref{leibniz}) is
\be\lb{last_term}
\frac {(x_k-y_k)(y_\ell-z_\ell) (v \otimes v) \cdot \nabla^2 A_k(y) A^\#_\ell(z)}{r_1^2 r_2^2} + \frac {(x_k-y_k)(y_\ell-z_\ell) \partial_\rho v \cdot \nabla A_k(y) A^\#_\ell(z)}{r_1^2 r_2^2}.
\ee
For $\rho>2r$ the integral of the first expression is bounded by $\|\nabla^2 A\|_\K \|A^\#\|_\K$ and that of the second expression is bounded by $\|\nabla A\|_\KK \|A^\#\|_\K$. This provides the $\widehat U(L^\infty)$ norm and the $\widehat U(\K^*_{log})$ norm is the same, but with the usual logarithmic modifications.

For $\rho<2r$, $v_1$ is less singular than the other two components of $v$. In the first expression the part that contains $v_1 \otimes v_1$ is bounded by $\|\nabla^2 A\|_\K \|A^\#\|_\K$ (with logarithmic modifications for $\widehat U(\K^*_{\log})$), while the parts that contain one copy of $v_1$ are bounded by
$$
\frac {\rho \sin \theta |\nabla^2 A(y)| |A^\#(z)|}{r_1 r_2 \sqrt{\rho^2-r^2}} \les \frac {|\nabla^2 A(y)| |A^\#(z)|}{\rho R}
$$
and eventually by $\|\nabla^3 A\|_{L^1} \|A^\#\|_\K$ in the $\widehat \U(L^\infty)$ norm and $\|\nabla^3 A\|_{L^1} \|A^\#\|_{\K_{\log}}$ in the $\widehat U(\K^*_{\log})$ norm.

In the second expression, $\partial_\rho v_1 = 0$, so we are left with the parts corresponding to $v_2$ and $v_3$.

From the expression involving $\nabla^2 A_k$ in (\ref{last_term}) we are left with
\be\lb{prim}
\frac {(x_k-y_k)(y_\ell-z_\ell)\rho^2\sin^2\theta\, \partial_R^2 A_k(y) A_\ell^\#(z)}{4r_1^2 r_2^2 (\rho^2-r^2)},
\ee
while from the expression involving $\nabla A_k$ in (\ref{last_term}) we are left with
\be\lb{second}
-\frac {(x_k-y_k)(y_\ell-z_\ell) r^2 \sin \theta \partial_R A_k(y) A_\ell^\#(z)}{2r_1^2 r_2^2 (\rho^2-r^2)^{3/2}}.
\ee
Again we distinguish three cases, according to whether $2 \leq k, \ell \leq 3$ or one or both of $k$ and $\ell$ are $1$.

If $2 \leq k, \ell \leq 3$, we get two factors of $\sqrt{\rho^2-r^2}$ in the numerator, canceling the same in the denominator. For (\ref{prim}), this leads to a bound of
$$
\frac {\rho^2 \sin^4 \theta |\partial_R^2 A_k(y)| |A_\ell^\#(z)|}{r_1^2 r_2^2} \les \frac {|\partial_R^2 A_k(y)| |A_\ell^\#(z)|}{\rho^2}
$$
and eventually $\|\nabla^2 A\|_\K \|A^\#\|_\K$ in $\widehat U(L^\infty)$ and $\|\nabla^2 A\|_{\K_{\log}} \|A^\#\|_{\K_{\log}}$ in $\widehat U(\K^*_{\log})$.

The $2 \leq k, \ell \leq 3$ terms in (\ref{second}) satisfy a bound of
$$
\frac {r^2 \sin^3 \theta |\partial_R A_k(y)| |A^\#_\ell(z)}{r_1^2 r_2^2 \sqrt{\rho^2-r^2}} \les \frac {|\partial_R A_k(y)| |A^\#_\ell(z)|}{\rho^2 R}
$$
and eventually $\|\nabla^2 A\|_{L^1} \|A^\#\|_\KK$ in $\widehat U(L^\infty)$ (using Lemmas \ref{lema2} and \ref{lema3}) and $\|\nabla^2 A\|_{L^1} \|A^\#\|_{\K_{2, \log}}$ in $\widehat U(\K^*_{\log})$.

If $k=1$ and $2 \leq \ell \leq 3$ or vice-versa, for (\ref{prim}) we have a factor of $R$ in the numerator which cancels one in the denominator:
$$
\les \frac {\rho^2 \sin^5 \theta |\partial_R^2 A_k(y)| |A_\ell^\#(z)|}{r_1 r_2^2 R} \les \frac {r_1^{3/2} r_2^{1/2} |\partial_R^2 A_k(y)| |A_\ell^\#(z)|}{\rho^3 R} \les \frac {|\partial_R^2 A_k(y)| |A_\ell^\#(z)|}{\rho R},
$$
which by Lemmas \ref{lema2} and \ref{lema3} leads to an estimate of $\|\nabla^3 A\|_{L^1} \|A^\#\|_\K$ in the $\widehat U(L^\infty)$ norm and an estimate of $\|\nabla^3 A\|_{L^1} \|A^\#\|_{\K_{\log}}$ in the $\widehat U(\K^*_{\log})$ norm.

However, there is no appropriate absolute value bound for (\ref{second}), which in this case equals
\be\lb{lb}
\frac {(\rho \cos \theta + r)r^2 \sin^2 \theta g(\phi) \partial_R A_1(y) A_\ell^\#(z)}{8r_1^2 r_2^2 (\rho^2-r^2)}.
\ee
Here we use the spare copy of (\ref{exp}), which produces a cancellation:
\be\lb{suma}
(\ref{lb}) + (\ref{exp}) = - \frac {(\rho \cos \theta + r) \sin^2 \theta g(\phi) \partial_R A_k(y) A_\ell(z)}{8r_1^2 r_2^2} \les \frac {|\nabla A_k(y)| |A_\ell(z)|}{\rho^2 r_2} \les \frac {|\nabla A_k(y)| |A_\ell(z)|}{r_1 r_2^2}.
\ee
This expression can be estimated in the usual way and admits a bound of $\|\nabla A\|_\K \|A^\#\|_\KK$ in $\widehat U(L^\infty)$ and one of $\|\nabla A\|_{\K_{\log}} \|A^\#\|_{\K_{2, \log^2}}$ in $\widehat U(\K^*_{\log})$.

When $k=\ell=1$, we replace $\partial^2_R A_1(y)$ in (\ref{prim}) by
$$
\frac 1 R \int_{y_0}^y \partial^2_R A_1(\eta) \dd \eta = \frac {2(\partial_R A_1(y) - \partial_R A_1(y_0))}{\sqrt{\rho^2-r^2} \sin \theta}
$$
and estimate the difference separately. Integrating in $\phi$, the term involving $\partial_R A_1(y_0)$ vanishes, because
$$
\int_0^{2\pi} \partial_R A_1(y_0) \dd \phi = 0.
$$
This leaves
\be\lb{tert}
\frac {(x_1-y_1)(y_1-z_1)\rho^2\sin\theta\, \partial_R A_1(y) A_1^\#(z)}{2r_1^2 r_2^2 (\rho^2-r^2)^{3/2}}.
\ee
Combining (\ref{tert}) with (\ref{second}), a factor of $\rho^2-r^2$ cancels, resulting in
$$
(\ref{tert})+(\ref{second}) = \frac {(x_1-y_1)(y_1-z_1)\sin \theta\, \partial_R A_1(y) A_1^\#(z)}{2r_1^2r_2^2 \sqrt{\rho^2-r^2}} \les \frac {\sin^2 \theta\, |\nabla A_1(y)| |A_1^\#(z)|}{r_1r_2 R} \les \frac {|\nabla A_1(y)| |A_1^\#(z)|}{\rho^2 R}.
$$
This is no more than $\|\nabla^2 A_1\|_{L^1} \|A^\#\|_\KK$ in the $\widehat \U(L^\infty)$ norm and $\|\nabla^2 A_1\|_{L^1} \|A^\#\|_{\K_{2, \log^2}}$ in the $\widehat \U(\K^*_{\log})$ norm.

The difference also has a gain:
$$
\bigg|\partial^2_R A_1(y) - \frac 1 R \int_{y_0}^y \partial^2_R A_1(\eta) \dd \eta\bigg| \leq \frac 1 R \int_{y_0}^y |\partial^2_R A_1(y) - \partial^2_R A_1(\eta)| \dd \eta \les R \|\nabla^3 A\|_{L^\infty}.
$$
Again, the $R$ factor cancels one of the powers of $R$ in the denominator, making the expression integrable.
We actually use the sharper estimate
$$
\bigg|\partial^2_R A_1(y) - \frac 1 R \int_{y_0}^y \partial^2_R A_1(\eta) \dd \eta\bigg| = \frac 1 R \bigg|\int_{y_0}^y \partial^2_R A_1(y) - \partial^2_R A_1(\eta) \dd \eta\bigg| \leq \frac 1 R \int_{y_0}^y (R-R_\eta) |\nabla^3 A(\eta)| \dd \eta.
$$
Here $R-R_\eta=|y-\eta|$, where $R=R_y=\sqrt{y_1^2+y_2^2}$ and $R_\eta=\sqrt{\eta_1^2+\eta_2^2}$ is the corresponding value for $\eta$.

Using this formula instead in (\ref{prim}) leads to
$$
\les \frac {\rho^2\sin^4 \theta (\int_{y_0}^y (R-R_\eta)|\nabla^2A(\eta)|\dd \eta) A_\ell^\#(z)}{r_1 r_2 R^3} \les \frac {(\int_{y_0}^y (R-R_\eta)|\nabla^2A(\eta)|\dd \eta) A_\ell^\#(z)}{\rho^2 R^3}.
$$
We first integrate in $y$ in cylindrical coordinates and are left with an integral in $\eta$. There is a factor of $R$ from using cylindrical coordinates, so the integral is
$$
\int_\eta^{y_1} \frac {R_y-R_\eta}{R_y^2},
$$
where the upper bound $y_1$ is comparable to $r$. The integral is at most $1+\log (r/R_\eta)$ and an extra factor of $1/R_\eta$ appears because computations so far were done in cylindrical coordinates:
$$
\frac {(1+|\log R_\eta|+|\log r|)|\nabla^3 A(\eta)| |A_\ell^\#(z)|}{\rho^2 R_\eta}.
$$
Using Lemma \ref{lema2log}, this is bounded by $\|\nabla^4 A\|_{L \log L} \|A^\#\|_\KK + \|\nabla^4 A\|_{L^1} \|A^\#\|_{\K_{2, \log}}$ in $\widehat \U(L^\infty)$ and by $\|\nabla^4 A\|_{L \log L} \|A^\#\|_{\K_{2, \log^2}}$ in $\widehat \U(\K^*_{\log})$.

This final term is another one that made the logarithmic modifications necessary in the previous computations.
\end{proof}

Recall (\ref{ttilde})
$$
\tilde T = [R_0(\lambda^2) \nabla A R_0(\lambda^2)](x, z).
$$
\begin{proposition}\lb{main_linear} If $A \in \dot W^{2, 1}$, then $\tilde T \in \widehat \U(\K, L^\infty) \cap \widehat \U(L^1, \K^*)$, with $\|\tilde T\| \les \|\nabla^2 A\|_{L^1}$.
\end{proposition}
Again note that $\K^* \subset \K^*_{log}$. If $V \in \K$, it follows that $T_2=\tilde T V \in \widehat \U(L^\infty) \cap \widehat U(\K^*)$, and if $V \in \K_{\log}$ then $T_2 \in \widehat U(\K^*_{\log})$ as well.
\begin{proof}
The integral kernel of $\tilde T(\lambda)$ is
\be\lb{t1}\begin{aligned}
\tilde T(\lambda)(x, z) &= \int_{\R^3} (\nabla_x R_0(\lambda^2)(x, y) \cdot A(y)) R_0(\lambda^2)(y, z) \dd y\\
&= \frac 1 {16 \pi^2} \int_{\R^3} \bigg(\frac {i\lambda e^{i\lambda|x-y|}}{|x-y|} - \frac {e^{i\lambda|x-y|}}{|x-y|^2} \bigg) \bigg(\frac {x-y}{|x-y|} \cdot A(y)\bigg) \frac {e^{i\lambda|y-z|}}{|y-z|} \dd y.
\end{aligned}\ee
	There are two terms to consider. Among other notations introduced in Section \ref{notations}, recall $r_1=|x-y|$, $\vec r_1=x-y$, $r_2=|y-z|$, $\rho=r_1+r_2$, $r=|x-z|$, $R=\sqrt{y_1^2+y_2^2}$, and
	$$
	\Sigma_\rho=\{y\in\R^3 \mid \rho = r_1+r_2 = |x-y|+|y-z|\}.
	$$
	The two terms are
	$$
	J_0 = \int_{\R^3} \frac {i\lambda e^{i\lambda \rho}(\vec r_1 \cdot A(y))}{r_1^2r_2} \dd y,\ J_1 = -\int_{\R^3} \frac {e^{i\lambda \rho}(\vec r_1 \cdot A(y))}{r_1^3r_2} \dd y,
	$$
	Let
	$$
	C_0(\rho, x, z) = \int_{\Sigma_\rho} \frac {\vec r_1 \cdot A(y)} {r_1^2r_2} J(y) \dd \mu,\ C_1(\rho, x, z) = -\int_{\Sigma_\rho} \frac {\vec r_1 \cdot A(y)} {r_1^3r_2} J(y) \dd \mu.
	$$
	Then $\widehat J_0 = -\partial_\rho C_0$ and $\widehat J_1 = C_1$.
	
	We start by proving the $\widehat U(\K, L^\infty)$ bound. Clearly
	\be\lb{model}
	\int_0^\infty |C_1(\rho, x, z)| \dd \rho \leq \int_{\R^3} \frac {|A(y)| \dd y}{|x-y|^2 |y-z|},
	\ee
	which is the integral kernel of a bounded operator from $\K$ to $L^\infty$, of norm $\|A\|_\KK$. Likewise
	\be\lb{c0}\begin{aligned}
	\partial_\rho C_0 &= \int_{\Sigma_\rho} \partial_\rho \bigg(\frac {\vec r_1 \cdot A(y)} {r_1^2r_2}\bigg) + \frac {\vec r_1 \cdot A(y)} {r_1^2r_2} \bigg( \frac 1 {2r_1} + \frac 1 {2r_2} \bigg) J(y) \dd \mu \\
	&= \int_{\Sigma_\rho} \frac {\partial_\rho (\vec r_1 \cdot A(y))} {r_1^2r_2} - \frac {\vec r_1 \cdot A(y)} {2r_1^3r_2} J(y) \dd \mu.
	\end{aligned}\ee
	We already know how to handle the second term, see (\ref{model}). Regarding the first term,
	$$
	\partial_\rho [(x_k-y_k)A_k(y)] = \partial_\rho (x_k-y_k) A_k(y) + (x_k-y_k) v \cdot \nabla A_k(y),
	$$
	so we break its contribution into six parts. For $k=1$, $\ds \partial_\rho (x_1-y_1) A_1(y) = -\frac {\cos \theta} 2 A_1(y)$ is similar to (\ref{model}). In the second term
	$$
	\frac {(x_1-y_1) v \cdot \nabla A_1(y)}{r_1^2 r_2},
	$$
	the contribution of $v_1$ is bounded by
	$$
	\frac {|\partial_1 A_1(y)|}{r_1 r_2}
	$$
	and gives rise to an integral kernel in $\B(\K, L^\infty)$ whose operator norm is bounded by $\|\nabla A\|_\K$.
	
	Write the contributions of $v_2$ and $v_3$ together as
	$$
	\frac {(x_1-y_1) \rho \sin \theta \partial_R A_1(y)}{r_1^2 r_2 \sqrt{\rho^2-r^2}}.
	$$
	When $\rho > 2r$, the same applies as above. When $\rho \leq 2r$, we use the bound
	$$
	\les \frac {\rho \sin^2 \theta \partial_R A_1(y)}{r_1 r_2 R} \les \frac {\partial_R A_1(y)}{\rho R},
	$$
	whose integral is at most $\|\nabla^2 A\|_{L^1} / \rho$ by Lemma \ref{lema2}. This gives rise to an integral kernel whose operator norm is bounded by $\|\nabla^2 A\|_{L^1}$.
	
	For $2 \leq k \leq 3$, the second term is
	$$
	\frac {\sqrt{\rho^2-r^2} \sin \theta g(\phi) v \cdot \nabla A_k(y)}{r_1^2 r_2}.
	$$
	For the $v_1$ component $\sqrt{\rho^2-r^2} \sin \theta \leq \min(r_1, r_2) \leq r_1$ and again we get a bound of $\|\nabla A\|_\K$. For the other two components we get
	$$
	\frac {\rho \sin^2 \theta g(\phi) \partial_R A_k(y)}{r_1^2 r_2} \les \frac {|\partial_R A_k(y)|}{r_1 \rho}
	$$
	resulting in the same bound.
	
	The first term is
	$$
	\frac {\rho \sin \theta g(\phi) A_k(y)}{r_1^2 r_2 \sqrt{\rho^2-r^2}} \les \frac {\rho \sin^2 \theta |A_k(y)|}{r_1^2 r_2 R} \les \frac {|A_k(y)|}{\rho r_1 R}.
	$$
	Lemmas \ref{lema2} and \ref{lema3} show that the $y$ integral is at most $\|\nabla^2 A\|_{L^1} / \rho$, leading to a bound of $\|\nabla^2 A\|_{L^1}$ for the operator norm.
	
	Finally, the same integral kernel bounds we have already proved show that these integral kernels are also in $\B(L^1, \K^*)$, with the same norms.
\end{proof}

Also consider
$$
T_3 = [R_0(\lambda^2) V R_0(\lambda^2) \nabla A](x, z).
$$
The computations are entirely analogous to those for $\tilde T$, so there is some repetition in the proof below.

\begin{proposition}\lb{main_linear2} If $V \in \dot W^{2, 1}$ and $A \in \K$, then $T_3 \in \widehat \U(L^\infty) \cap \widehat U(\K^*)$, with $\|T_3\| \les \|V\|_{\dot W^{2, 1}} \|A\|_\K$. If in addition $A \in \K_{\log}$, then $T_3 \in \widehat U(\K^*_{\log})$ with $\|T_3\| \les \|V\|_{\dot W^{2, 1}} \|A\|_{\K_{\log}}$.
\end{proposition}
This estimate is suboptimal and probably $V \in \K$ is sufficient for the main results, but for the sake of brevity we handle this term in the same way as the others.
\begin{proof}
	The integral kernel of $T_3(\lambda)$ is
	$$\begin{aligned}
		T_3(\lambda)(x, z) &= \int_{\R^3} R_0(\lambda^2)(x, y) V(y)  (\nabla_y R_0(\lambda^2)(y, z) \cdot A(z)) \dd y\\
		&= \frac 1 {16 \pi^2} \int_{\R^3} \frac {e^{i\lambda|x-y|}}{|x-y|} V(y) \bigg(\frac {i\lambda e^{i\lambda|y-z|}}{|y-z|} - \frac {e^{i\lambda|y-z|}}{|y-z|^2} \bigg) \bigg(\frac {y-z}{|y-z|} \cdot A(z)\bigg) \dd y.
	\end{aligned}$$
	For notations the reader is referred to Section \ref{notations}. There are two terms to consider, namely
	$$
	K_0 = \int_{\R^3} \frac {i\lambda e^{i\lambda \rho}V(y)(\vec r_2 \cdot A(z))}{r_1r_2^2} \dd y,\ K_1 = -\int_{\R^3} \frac {e^{i\lambda \rho}V(y)(\vec r_2 \cdot A(z))}{r_1r_2^3} \dd y,
	$$
	Let
	$$
	D_0(\rho, x, z) = \int_{\Sigma_\rho} \frac {V(y)(\vec r_2 \cdot A(z))} {r_1 r_2^2} J(y) \dd \mu,\ D_1(\rho, x, z) = -\int_{\Sigma_\rho} \frac {V(y)(\vec r_2 \cdot A(z))} {r_1r_2^3} J(y) \dd \mu.
	$$
	Then $\widehat K_0 = -\partial_\rho D_0$ and $\widehat K_1 = D_1$.
	
	We begin with the proof of the $\widehat U(L^\infty)$ bound. Clearly
	\be\lb{model2}
	\int_0^\infty |D_1(\rho, x, z)| \dd \rho \leq \int_{\R^3} \frac {|V(y)||A(z)| \dd y}{|x-y||y-z|^3}
	\ee
	and when integrating in $z$ as well we obtain a bound of $\|V\|_\K \|A\|_\KK$. Likewise
	$$\begin{aligned}
	\partial_\rho D_0 &= \int_{\Sigma_\rho} \partial_\rho \bigg(\frac {V(y)(\vec r_2 \cdot A(z))} {r_1r_2^2}\bigg) + \frac {V(y)(\vec r_2 \cdot A(z))} {r_1r_2^2} \bigg( \frac 1 {2r_1} + \frac 1 {2r_2} \bigg) J(y) \dd \mu \\
	&= \int_{\Sigma_\rho} \frac {\partial_\rho [V(y)(\vec r_2 \cdot A(z))]} {r_1r_2^2} - \frac {V(y)(\vec r_2 \cdot A(z))} {2r_1r_2^3} J(y) \dd \mu.
	\end{aligned}$$
	The second term is similar to (\ref{model2}). Regarding the first term,
	$$
	\partial_\rho [V(y)(y_k-z_k)A_k(z)] = V(y) \partial_\rho (y_k-z_k) A_k(z) + (v \cdot \nabla V) (y_k-z_k) A_k(z).
	$$
	Again we consider six different expressions, each term for $1 \leq k \leq 3$.
	
	For $k=1$, $\ds V(y) \partial_\rho (y_1-z_1) A_1(z) = V(y) \frac {\cos \theta} 2 A_1(z)$ is similar to (\ref{model2}). In the second term
	$$
	\frac {(v \cdot \nabla V(y))(y_1-z_1) A_1(z)}{r_1 r_2^2},
	$$
	the contribution of $v_1$ is bounded by
	$$
	\frac {|\nabla V(y)||A_1(y)|}{r_1 r_2},
	$$
	which leads to a bound of $\|\nabla V\|_\K \|A\|_\K$.
	
	Write the contributions of $v_2$ and $v_3$ together as
	$$
	\frac {\rho \sin \theta \partial_R V(y) (y_1-z_1) A_1(z)}{r_1 r_2^2 \sqrt{\rho^2-r^2}}.
	$$
	When $\rho > 2r$, the same applies as above. When $\rho \leq 2r$, we use the bound
	$$
	\les \frac {\rho \sin^2 \theta \partial_R V(y) A_1(z)}{r_1 r_2 R} \les \frac {\partial_R V(y) A_1(z)}{\rho R},
	$$
	leading to $\|\nabla^2 V\|_{L^1} \|A\|_\K$ by Lemma \ref{lema2}.
	
	For $2 \leq k \leq 3$, the second term is
	$$
	\frac {(v\cdot \nabla V) \sqrt{\rho^2-r^2} \sin \theta g(\phi) A_k(z)}{r_1 r_2^2}.
	$$
	For the $v_1$ component $\sqrt{\rho^2-r^2} \sin \theta \leq \min(r_1, r_2) \leq r_2$ and again we get a bound of $\|\nabla V\|_\K \|A\|_\K$. For the other two components we get
	$$
	\frac {\rho \sin^2 \theta g(\phi) \partial_R V A_k(z)}{r_1 r_2^2} \les \frac {|\partial_R V(y)| |A_k(z)|}{r_2 \rho}
	$$
	resulting in the same bound.
	
	The first term is
	$$
	\frac {V(y) \rho \sin \theta g(\phi) A_k(z)}{r_1 r_2^2 \sqrt{\rho^2-r^2}} \les \frac {|V(y)| \rho \sin^2 \theta |A_k(z)|}{r_1 r_2^2 R} \les \frac {|V(y)| |A_k(z)|}{\rho R r_2}.
	$$
	Lemmas \ref{lema2} and \ref{lema3} lead to a bound of $\|\nabla^2 V\|_{L^1} \|A\|_\K$.
	
	Finally, the same integral kernel bounds proved above can also be used to prove the $\widehat U(\K^*)$ and $\widehat U(\K^*_{\log})$ boundedness, with the appropriate norm.
\end{proof}

We also need appropriate bounds for $\partial_\lambda T$ and $\partial_\lambda \tilde T$.

\begin{proposition}\lb{prop14} Assume that $A \in \K_{2, \log^2}$ is such that $\nabla A \in \K_{\log} \cap \K_{2, \log^2}$, $\nabla^2 A \in \K_{2, \log^2} \cap L^1$, $\nabla^3 A, \nabla^4 A \in L \log L$, and let $A^\# \in \K_{\log} \cap \K_{2, \log^2}$, $V \in \K \cap \dot W^{2, 1}$, and $V^\# \in \K_{\log}$. Then the operator $\partial_\lambda T$ belongs to $\widehat \U(\K^*_{\log}, L^\infty)$, with norm
$$
\|\partial_\lambda T\| \les Bil(A, A^\#) + \|A\|_{\K_{\log}} \|A^\#\|_{\K_{2, \log^2}} + \|A\|_{\dot W^{2, 1}} \|V^\#\|_{\K_{\log}} + \|V\|_{\dot W^{2, 1}} \|A^\#\|_{\K_{\log}} + \|V\|_\K \|V^\#\|_{\K_{\log}}.
$$
Furthermore, $\partial_\lambda \tilde T \in \U(L^1, L^\infty)$ with $\|\partial_\lambda \tilde T\| \les \|\nabla^2 A\|_{L^1}$.
\end{proposition}
Recall that the expression $Bil$ is defined in (\ref{bil}).
\begin{proof} Recall that $T$ is the sum of four terms, from $T_1$ to $T_4$, as per (\ref{t14}). We analyze each of them, one by one.

Since
$$
T_1(\lambda) = R_0(\lambda^2) \nabla A R_0(\lambda^2) \nabla A^\#,
$$
it follows that
$$
\partial_\lambda T_1 = \partial_\lambda R_0(\lambda^2) \nabla A R_0(\lambda^2) \nabla A^\# + R_0(\lambda^2) \nabla A \partial_\lambda R_0(\lambda^2) \nabla A^\# = T_{11} + T_{12}.
$$
Note that
$$
\partial_\lambda R_0(\lambda^2)(x, y) = \frac i {4\pi} e^{i\lambda|x-y|},\ \nabla \partial_\lambda R_0(\lambda^2)(x, y) = -\frac 1 {4\pi} \lambda e^{i\lambda|x-y|} \frac {x-y}{|x-y|}.
$$
First consider $T_{11}$. The integral kernel of $T_{11}(\lambda)$ is
\be\lb{t11}\begin{aligned}
T_{11}(\lambda)(x, z) &= -\frac 1 {16 \pi^2} \int_{\R^3} \lambda e^{i\lambda|x-y|} \bigg(\frac {x-y}{|x-y|} \cdot A(y)\bigg) \bigg(\frac {i\lambda e^{i\lambda|y-z|}}{|y-z|} - \frac {e^{i\lambda|y-z|}}{|y-z|^2} \bigg) \bigg(\frac {y-z}{|y-z|} \cdot A^\#(z)\bigg) \dd y.
\end{aligned}\ee
This is again the sum of two terms
$$
L_1 = \int_{\R^3} \frac {i\lambda^2 e^{i\lambda \rho}(\vec r_1 \cdot A(y))(\vec r_2 \cdot A^\#(z))}{r_1 r_2^2} \dd y,\ L_2 = -\int_{\R^3} \frac {\lambda e^{i\lambda \rho}(\vec r_1 \cdot A(y))(\vec r_2 \cdot A^\#(z))}{r_1 r_2^3} \dd y.
$$
Let
$$
E_1(\rho) = \int_{\Sigma_\rho} \frac {(\vec r_1 \cdot A(y))(\vec r_2 \cdot A^\#(z))}{r_1 r_2^2} J(y) \dd \mu,\ E_2(\rho) = \int_{\Sigma_\rho} \frac {(\vec r_1 \cdot A(y))(\vec r_2 \cdot A^\#(z))}{r_1 r_2^3} J(y) \dd \mu.
$$
Then $\widehat L_1 = -i \partial_\rho^2 E_1$, $\widehat L_2 = -i \partial_\rho E_2$ and, letting $b_0=(\vec r_1 \cdot A(y))(\vec r_2 \cdot A^\#(z))$,
\be\lb{e2}\begin{aligned}
\partial_\rho E_2 &= \int_{\Sigma_\rho} \partial_\rho \frac {b_0}{r_1 r_2^3} + \frac {b_0}{r_1 r_2^3} \bigg(\frac 1 {2r_1} + \frac 1 {2r_2}\bigg) J(y) \dd \mu \\
&= \int_{\Sigma_\rho} \frac {\partial_\rho b_0}{r_1 r_2^3} - \frac {b_0}{r_1 r_2^4} J(y) \dd \mu
\end{aligned}\ee
\be\lb{e1}\begin{aligned}
\partial^2_\rho E_1 &= \partial_\rho \int_{\Sigma_\rho} \partial_\rho \frac {b_0}{r_1 r_2^2} + \frac {b_0}{r_1 r_2^2} \bigg(\frac 1 {2r_1} + \frac 1 {2r_2}\bigg) J(y) \dd \mu \\
&= \partial_\rho \int_{\Sigma_\rho} \frac {\partial_\rho b_0}{r_1 r_2^2} - \frac {b_0}{2r_1 r_2^3} J(y) \dd \mu \\
&= \int_{\Sigma_\rho} \frac {\partial^2_\rho b_0}{r_1 r_2^2} - \frac {\partial_\rho b_0}{2 r_1^2 r_2^2} - \frac {\partial_\rho b_0}{r_1 r_2^3} + \frac {\partial_\rho b_0}{r_1 r_2^2} \bigg(\frac 1 {2r_1} + \frac 1 {2r_2}\bigg) J(y) \dd \mu - \frac 1 2  {\partial_\rho E_2} \\
&= \int_{\Sigma_\rho} \frac {\partial^2_\rho b_0}{r_1 r_2^2} - \frac {\partial_\rho b_0}{r_1 r_2^3} + \frac {b_0}{2r_1 r_2^4} J(y) \dd \mu.
\end{aligned}\ee
We are left with
\be\lb{left}
\int_{\Sigma_\rho} \frac {\partial_\rho^2 b_0}{r_1 r_2^2} - \frac {b_0}{2r_1 r_2^4} J(y) \dd \mu.
\ee
Importantly, the term involving $\ds \frac{\partial_\rho b_0}{r_1r_2^3}$, which is too singular for our estimates, cancels in the end.

We conduct a term-by-term analysis, keeping in mind the analogies to the proof of Proposition \ref{main_bilinear}.
	
	Some basic building blocks are integral kernels of the form
	$$
	F_1(x, y) = \frac {f(y)}{|x-y|^2} \in \B(\K^*_{\log}),\ f \in \K_{2, \log^2},
	$$
	and
	$$
	F_2(x, y) = \frac {f(y)}{|x-y|} \in \B(\K^*_{\log}, L^\infty),\ f \in \K_{2, \log^2}.
	$$
	
	The (\ref{e2}) term
	$$
	\frac {(\vec r_1 \cdot A(y))(\vec r_2 \cdot A^\#(z))}{r_1 r_2^4} = \sum_{k, \ell=1}^3 \frac {(x_k-y_k) A_k(y) (y_\ell-z_\ell) A_\ell^\#(z))}{r_1 r_2^4}.
	$$
	is similar to (\ref{b1r1}). The most singular contribution is from the case when $k=1$.
	
	If the denominator powers were balanced, $r_1^{-2} r_2^{-3}$, it would be bounded by a combination of the basic building blocks above, with norm at most $\|A\|_{\K_{2, \log^2}} \|A^\#\|_{\K_{2, \log^2}}$.
	
	Again we distinguish two situations,
	$$
	D_0=\{y \mid r_2>r_{20} = \min(\frac r {100}, 1)\} \text{ and } D=\{y \mid r_2 \leq r_{20}\}.
	$$
	Note the presence of $1$ in this definition of $r_{20}$, in order to avoid logarithmic growth in space.
	
	If $\ds\frac r {100} > 1$, then the bound is $1$ and in region $D_0$ the integrand is bounded by, for example,
	$$
	\frac {|A_1(y)| |A^\#_\ell(z)|}{r_2^2},
	$$
	where $\ds \frac {|A^\#_\ell(z)|}{|y-z|^2} \in \B(\K_{\log})$ of norm at most $\|A^\#\|_{\K_{2, \log^2}}$.
	This produces a bound of $\|A\|_{\K_{\log}} \|A^\#\|_{\K_{2, \log^2}}$.
	
	On the other hand, if $\ds\frac r {100} < 1$, we obtain $\|A\|_{\K_{2, \log^2}} \|A^\#\|_{\K_{2, \log^2}}$.
	
	Inside region $D$ we replace $A_k(y)$ by $A_k(z)$ under the integral, then evaluate the integral and the difference separately, as in our estimate of (\ref{b1r1}). Here we estimate the difference by
	$$
	\les \int_{D} \frac {|\nabla A_k(y)| |\log r_2| |A_1^\#(z)| \dd \eta}{r_2^2},
	$$
	leading to a $\|\nabla A\|_{\K_{\log}} \|A^\#\|_{\K_{2, \log^2}}$ bound in $\widehat U(\K^*_{\log}, L^\infty)$.
	
	The integral itself is zero due to symmetries for $2 \leq k \leq 3$. For $k=1$, after replacing $A_1(y)$ by $A_1(z)$, we get a bound of $\|A\|_{L^\infty} \|A^\#\|_{\K_{2, \log^2}}$.

For the other (\ref{e1}) term
$$
\frac {\partial^2_\rho[(\vec r_1 \cdot A(y))(\vec r_2 \cdot A^\#(z))]}{r_1 r_2^2} = \sum_{k, \ell=1}^3 \frac {\partial^2_\rho[(x_k-y_k)A_k(y) (y_\ell-z_\ell)A^\#(z)]}{r_1 r_2^2}
$$
we again use the Leibniz rule (\ref{leibniz}). The estimates are self-contained and entirely analogous to those for the contribution of (\ref{leibniz}) in the proof of Proposition \ref{main_bilinear}, including the delicate cancellations that take place there. The only difference is having one less power of $r_1$ in the denominator here, which makes the terms bounded in $\B(\K^*_{\log}, L^\infty)$.

Next, consider $T_{12}$. The integral kernel of $T_{12}(\lambda)$ is
$$\begin{aligned}
T_{12}(\lambda)(x, z) &= -\frac 1 {16 \pi^2} \int_{\R^3} \bigg(\frac {i\lambda e^{i\lambda|x-y|}}{|x-y|} - \frac {e^{i\lambda|x-y|}}{|x-y|^2} \bigg) \bigg(\frac {x-y}{|x-y|} \cdot A(y)\bigg) \lambda e^{i\lambda|y-z|} \bigg(\frac {y-z}{|y-z|} \cdot A^\#(z)\bigg) \dd y.
\end{aligned}$$

This is also the sum of two terms
$$
L_3 = \int_{\R^3} \frac {i\lambda^2 e^{i\lambda \rho}(\vec r_1 \cdot A(y))(\vec r_2 \cdot A^\#(z))}{r_1^2 r_2} \dd y,\ L_4 = -\int_{\R^3} \frac {\lambda e^{i\lambda \rho}(\vec r_1 \cdot A(y))(\vec r_2 \cdot A^\#(z))}{r_1^3 r_2} \dd y.
$$
Let
$$
E_3(\rho) = \int_{\Sigma_\rho} \frac {(\vec r_1 \cdot A(y))(\vec r_2 \cdot A^\#(z))}{r_1^2 r_2} J(y) \dd \mu,\ E_4(\rho) = \int_{\Sigma_\rho} \frac {(\vec r_1 \cdot A(y))(\vec r_2 \cdot A^\#(z))}{r_1^3 r_2} J(y) \dd \mu.
$$
Then $\widehat L_3 = -i \partial_\rho^2 E_3$, $\widehat M_2 = -i \partial_\rho E_4$ and, with $b_0=(\vec r_1 \cdot A(y))(\vec r_2 \cdot A^\#(z))$, following a computation similar to (\ref{e2}--\ref{e1}), we are left with
$$
\int_{\Sigma_\rho} \frac {\partial_\rho^2 b_0}{r_1^2 r_2} - \frac {b_0}{2r_1^4 r_2} J(y) \dd \mu.
$$
Henceforth the estimates are entirely analogous to the ones above for (\ref{left}).

Now consider
$$
\partial_\lambda \tilde T(\lambda) = \partial_\lambda R_0(\lambda^2) \nabla A R_0(\lambda^2) + R_0(\lambda^2) \nabla A \partial_\lambda R_0(\lambda^2) = \tilde T_1 + \tilde T_2.
$$
The integral kernel of $\tilde T_1$ is given by
\be\lb{ttilde1}\begin{aligned}
\tilde T_1(\lambda)(x, z) &= -\frac 1 {16 \pi^2} \int_{\R^3} \lambda e^{i\lambda|x-y|} \bigg(\frac {x-y}{|x-y|} \cdot A(y)\bigg) \frac {e^{i\lambda|y-z|}}{|y-z|} \dd y.
\end{aligned}\ee
This consists of just one term, which we can write using our notations as
$$
M_1(\lambda)=\int_{\R^3} \frac {\lambda e^{i\lambda\rho}(\vec r_1 \cdot A(y))}{r_1 r_2} \dd y.
$$
Let
$$
F_1(\rho) = \int_{\Sigma_\rho} \frac {\vec r_1 \cdot A(y)}{r_1 r_2} J(y) \dd \mu.
$$
Then $\widehat M_1 = i \partial_\rho F_1$, which is
\be\lb{s1}
\partial_\rho F_1 = \int_{\Sigma_\rho} \partial_\rho \frac {\vec r_1 \cdot A(y)}{r_1 r_2} + \frac {\vec r_1 \cdot A(y)}{r_1 r_2} \bigg(\frac 1 {2r_1} + \frac 1 {2r_2}\bigg) J(y) \dd \mu = \int_{\Sigma_\rho} \frac {\partial_\rho(\vec r_1 \cdot A(y))}{r_1 r_2} J(y) \dd \mu.
\ee
This is entirely similar to the first term in (\ref{c0}), except that here we have one less power of $r_1$ in the denominator, making this operator belong to $\U(L^1, L^\infty)$ instead.

In the integral kernel of $\tilde T_2$
$$\begin{aligned}
\tilde T_2(\lambda)(x, z) &= \frac i {16 \pi^2} \int_{\R^3} \bigg(\frac {i\lambda e^{i\lambda|x-y|}}{|x-y|} - \frac {e^{i\lambda|x-y|}}{|x-y|^2} \bigg) \bigg(\frac {x-y}{|x-y|} \cdot A(y)\bigg) e^{i\lambda|y-z|} \dd y,
\end{aligned}$$
there are two terms
$$
M_2 = \int_{\R^3} \frac {i\lambda e^{i\lambda\rho}(\vec r_1 \cdot A(y))}{r_1^2} \dd y,\ M_3 = -\int_{\R^3} \frac {e^{i\lambda\rho}(\vec r_1 \cdot A(y))}{r_1^3} \dd y.
$$
One has that $\widehat M_2 = -\partial_\rho F_2$, where
\be\lb{s2}
\partial_\rho F_2 = \int_{\Sigma_\rho} \partial_\rho \frac {\vec r_1 \cdot A(y)}{r_1^2} + \frac {\vec r_1 \cdot A(y)}{r_1^2} \bigg(\frac 1 {2r_1} + \frac 1 {2r_2}\bigg) J(y) \dd \mu = \int_{\Sigma_\rho} \frac {\partial_\rho (\vec r_1 \cdot A(y))}{r_1^2} - \frac {\vec r_1 \cdot A(y)}{2r_1^3} + \frac {\vec r_1 \cdot A(y)}{2r_1^2 r_2} J(y) \dd \mu.
\ee
The first term is again similar to the first term in (\ref{c0}), except for having one less power of $r_2$ in the denominator, making this operator belong to $\U(L^1, L^\infty)$ instead. The other two terms are bounded by $\|A\|_\KK$ in $\U(L^1, L^\infty)$.

Consequently $\partial_\lambda \tilde T \in \U(L^1, L^\infty)$ with $\|\partial_\lambda \tilde T\| \les \|\nabla^2 A\|_{L^1}$ and
$$
\partial_\lambda T_2 = \partial_\lambda \tilde T V^\# \in \U(\K^*_{\log}, L^\infty),
$$
with a norm bounded by $\|\nabla^2 A\|_{L^1} \|V^\#\|_{\K_{\log}}$.

As seen in Propositions \ref{main_linear} and \ref{main_linear2}, $T_3$ is entirely similar to $T_2$, with
$$
\|\partial_\lambda T_3\|_{\U(\K^*_{\log}, L^\infty)} \les \|\nabla^2 V\|_{L^1} \|A^\#\|_{\K_{\log}}.
$$

Finally, a simple computation shows that
$$
\|\partial_\lambda T_4\|_{\U(\K^*_{\log}, L^\infty)} \les \|V\|_{\K} \|V^\#\|_{\K_{\log}}.
$$
\end{proof}

\subsection{Wiener's theorem and the proof of the main results}

Recall $T$ is given by (\ref{deft}) and $R$ is the perturbed resolvent.

\begin{proposition}\lb{prop_wiener} Assume that $A \in X_0$ and $V \in Y_0$ and that $0$ is neither an eigenvalue nor a resonance for $H=-\Delta+U$ or for $\tilde H = -\Delta-U$. Then
\be\lb{est1}
(I-T)^{-1} \in \widehat U(L^\infty) \cap \widehat U(\K^*_{\log}).
\ee
Furthermore $R(\lambda^2) \in \widehat \U(\K, L^\infty) \cap \widehat \U(L^1, \K^*_{\log})$ and $\partial_\lambda R(\lambda^2) \in \widehat \U(L^1, L^\infty)$.
\end{proposition}

\begin{proof}
We write $T$ as a sum of four terms as in (\ref{t14}). For the last term $T_4$,
$$
\frac {|V(y)| |V(z)|} {r_1 r_2}
$$
admits a bound of $\|V\|_\K^2$. Then $T \in \widehat \U(L^\infty) \cap \widehat \U(\K^*_{\log})$ by Propositions \ref{main_bilinear}, \ref{main_linear}, and \ref{main_linear2}.

If $A$ and $V$ are small in norm, then we can obtain the inverse of $I-T$ by a series expansion. Otherwise, we use Wiener's Theorem \ref{wiener}.

We next check that the conditions of Theorem \ref{wiener} are met. For that purpose, since $A \in X_0$ and $V \in Y_0$, it helps to approximate both by sequences $(A_n)_n$ and $(V_n)_n$ of test functions. If a desirable property below holds for each approximation, it also holds for the limit.

Thus, the compact support of $A_n$ and $V_n$ implies that for each fixed $x$ and $y$ $\widehat T_n(\rho, x, y)$ becomes $0$ for sufficiently large $\rho$ (when $\Sigma_\rho$ no longer intersects the supports of $A$ and $V$) and $\|\widehat T_n(\rho, x, y)\|_{L^1_\rho}$ goes to $0$ as $x \to \infty$ or $y \to \infty$. By dominated convergence, this implies property $(*)$.

Likewise, for smooth $A_n$ and $V_n$ it is easy to see that $\widehat T_n(\rho, x, y)$ is continuous in all three variables. By dominated convergence, this implies property $(**)$.

Finally, the Arzela--Ascoli criterion can be used to prove the compactness of $T_n(\lambda)$ for each $\lambda$ (it takes values in $C_0$).

By using test function approximations to $A$ and $V$, all three properties transfer to $T$.

Suppose $I-T(\lambda) \in \B(L^\infty) \cap \B(\K^*_{\log})$ is not invertible for some $\lambda \in \R$; then by Fredholm's alternative there exists $g \in L^\infty \cap \K^*_{\log}$ such that
\be\lb{double}
g = T g = R_0(\lambda^2)UR_0(\lambda^2)U g.
\ee
Bootstrapping, we obtain that $g, R_0(\lambda^2) U g \in L^\infty \cap L^{3, \infty}$. Indeed, the integral kernels of $R_0$ and $R_0 \nabla$ consist of convolution with functions that decay like $|x|^{-1}$ for the former and both $|x|^{-1}$ and $|x|^{-2}$ for the latter. We place $Vg$ in $L^1 \cap \K$, which requires $V \in \K_{\log}$:
$$
\|V g\|_{L^1} \leq \|V\|_{\K_{\log}} \|g\|_{\K^*_{\log}},\ \|V g\|_{\K} \leq \|V\|_\K \|g\|_{L^\infty}
$$
and $Ag$ in $L^1 \cap \K \cap \KK$, which can be done by requiring that $A \in \K_{\log} \cap \KK$:
$$
\|Ag\|_{L^1} \leq \|A\|_{\K_{\log}} \|g\|_{\K^*_{\log}},\ \|A g\|_{\K} \leq \|A\|_\K \|g\|_{L^\infty},\ \|A g\|_{\KK} \leq \|A\|_\KK \|g\|_{L^\infty}.
$$
Then the convolution product $R_0(\lambda^2) U g = R_0(\lambda^2)(\nabla (Ag) + Vg)$ will be in $L^\infty \cap \K^*$. Iterating we obtain the same for $g$ itself, by (\ref{double}).

Then either $g=R_0(\lambda^2) U g$ or $g-R_0(\lambda^2)Ug \ne 0$ is in the kernel of $I+R_0(\lambda^2)U$. The former situation is the same as the latter, except with $U$ replaced by $-U$, so with no loss of generality consider the second situation: there exists $g \in L^\infty \cap \K^*$, $g \ne 0$, such that $g=-R_0(\lambda^2) U g$.

If $\lambda=0$, this possibility is already excluded by Assumption \ref{a1}, so let $\lambda \ne 0$. For simplicity, the rest of the proof uses Lebesgue spaces. Note that
$$
(-\Delta) R_0(\lambda^2) = I + \lambda^2 R_0(\lambda^2).
$$
Consequently $(-\Delta) R_0(\lambda^2) Ag \in L^{3, \infty}$, as $Ag \in L^1 \cap L^{3,\infty}$:
$$
\|Ag\|_{L^1} \les \|A\|_\K \|g\|_{\K^*},\ \|Ag\|_{L^{3, \infty}} \les \|A\|_{L^{3, \infty}} \|g\|_{L^\infty},\ R_0(\lambda^2) \in \B(L^1, L^{3, \infty}).
$$
Then $R_0(\lambda^2) \nabla(Ag) \in |\nabla|^{-1} L^{3, \infty}$, by elliptic regularity, so $\nabla R_0(\lambda^2) \nabla(Ag) \in L^{3, \infty}$.

Likewise, $\nabla R_0(\lambda^2) Vg \in L^{3, \infty}$, since
$$
\|Vg\|_{L^{3/2, \infty}} \leq \|V\|_{L^{3/2, \infty}} \|g\|_{L^\infty},\ \|V g\|_{L^1} \leq \|V\|_\K \|g\|_{\K^*},
$$
and $\nabla R_0(\lambda^2)$ consists of convolution with the sum of two terms, one dominated by $|x|^{-1}$ and the other by $|x|^{-2}$.

We obtain that
$$
\nabla g = \nabla R_0(\lambda^2) \nabla(Ag) + \nabla R_0(\lambda^2) V g \in L^{3, \infty}.
$$

Hence $A \nabla g \in L^{3/2, 1} \cap L^1$, as $A \in L^{3, 1} \cap L^{3/2, 1}$. But then
$$
f=Ug=A\nabla g + (\nabla A) g + V g \in L^{3/2, 1} \cap L^1
$$
as well, since $\|(\nabla A + V) g\|_{L^{3/2, 1} \cap L^1} \leq \|g\|_{L^\infty \cap L^{3, \infty}} (\|\nabla A\|_{L^{3/2, 1}} + \|V\|_{L^{3/2, 1}})$.

Then the pairing $\langle f, g \rangle = \langle Ug, g \rangle$ is well-defined (between an $L^{3, \infty}$ and an $L^{3/2, 1}$ function) and real-valued, since $U$ is self-adjoint. By Agmon's \cite{Agm} bootstrap argument
$$
-\langle g, U g \rangle = \langle R_0(\lambda^2) U g, U g \rangle
$$
is real, hence its imaginary part is $0$. The imaginary part is given by $\langle [R_0(\lambda^2)-R_0^*(\lambda^2)] U g, U g \rangle$, where $R_0^*(\lambda^2)=R_0((-\lambda)^2)$ (which is not the same as $R_0(\lambda^2)$, in the sense that $R_0((\cdot)^2)$ is defined on the Riemann surface of the square root) and
$$
\langle R_0(\lambda^2)-R_0^*(\lambda^2) f, g \rangle = c \int_{S^2} \widehat f(\lambda \omega) \ov{\widehat g(\lambda \omega)}\dd \omega.
$$
It follows that $\widehat f \mid_{\lambda S^2} = 0$.

Then by Corollary 13 in \cite{GolSch}, as $\widehat f$ vanishes on a sphere and $f \in L^1$, it follows that $g =-R_0(\lambda^2)f\in L^2$. It is easily seen that $g \ne 0$ solves the eigenfunction equation $Hg=\lambda^2 g$ for the eigenvalue $\lambda^2$. This contradicts the absence of eigenvalues in $(0, \infty)$, proved in \cite{KocTat} under an assumption weaker than $A \in L^3$, $V \in L^{3/2}$, which is the case here.

The contradiction shows that $I-T(\lambda)$ is invertible for all $\lambda \in \R$. We can apply Wiener's Theorem \ref{wiener} and obtain (\ref{est1}).

Next, for $\tilde T$ given by (\ref{ttilde}),
$$
R=(I-T)^{-1} (R_0 - \tilde T - R_0 V R_0).
$$
Since $R_0(\lambda^2)$, $\tilde T(\lambda)$ (by Proposition \ref{main_linear}), and $R_0(\lambda^2) V R_0(\lambda^2)$ are in $\widehat \U(\K, L^\infty) \cap \widehat \U(L^1, \K^*)$, it follows that $R(\lambda^2) \in \widehat \U(\K, L^\infty) \cap \widehat \U(L^1, \K^*_{\log})$.

Likewise
$$
\partial_\lambda R = (I-T)^{-1} \partial_\lambda T (I-T)^{-1} (R_0 - \tilde T - R_0 V R_0) + (I-T)^{-1} (\partial_\lambda R_0 - \partial_\lambda \tilde T - \partial_\lambda (R_0 V R_0)).
$$

Since $\partial_\lambda T \in \widehat \U(\K^*_{\log}, L^\infty)$, by Proposition \ref{prop14}, and $\partial_\lambda R_0$, $\partial_\lambda \tilde T$ (again by Proposition \ref{prop14}), and $\partial_\lambda (R_0 V R_0)$ are in $\widehat U(L^1, L^\infty)$, it follows that $\partial_\lambda R(\lambda^2) \in \widehat U(L^1, L^\infty)$, as claimed.
\end{proof}

\begin{corollary} Under the same conditions
\be\lb{one}
\sup_{x, y} \int_0^\infty t \bigg|\frac {\sin(t \sqrt H) P_{ac}}{\sqrt H}(x, y)\bigg| \dd t < \infty
\ee
and
\be\lb{two}
\int_0^\infty \bigg|\frac {\sin(t \sqrt H) P_{ac}}{\sqrt H}(x, y)\bigg| \dd t \les \frac 1 {|x-y|}.
\ee
\end{corollary}
The reason why we did not use this formalism from the beginning is that $\frac {\sin(t \sqrt{-\Delta})}{\sqrt{-\Delta}} \nabla$ contains the derivative of a measure, in three dimensions, and the computations for determining whether the product of distributions is valid, as per H\"{o}rmander, are unwieldy.
\begin{proof}
Start from the functional calculus formula, for $t \geq 0$
$$\begin{aligned}
\frac {\sin(t \sqrt H) P_{ac}}{\sqrt H} &= \frac 1 {2\pi i} \int_0^\infty \frac {\sin(t \sqrt \lambda)}{\sqrt \lambda} (R(\lambda-i0)-R(\lambda+i0)) P_{ac} \dd \lambda \\
&= \frac i \pi \int_\R \sin(t\lambda) R((\lambda+i0)^2) P_{ac} \dd \lambda.
\end{aligned}$$
As $R((\lambda+i0)^2) P_{ac}$ is analytic and bounded in the upper half-plane, for $t \geq 0$ 
$$
\int_\R e^{it\lambda} R((\lambda+i0)^2) P_{ac} = 0,
$$
so
$$
\chi_{t \geq 0} \frac {\sin(t \sqrt H) P_{ac}}{\sqrt H} = \frac 1 {2\pi} \int_\R e^{-it\lambda} R((\lambda+i0)^2) P_{ac} \dd \lambda.
$$
In other words, as seen in \cite{BecGol2}, our statements about the Fourier transform of $R(\lambda^2)$ actually refer to the sine propagator for the wave equation. We write (\ref{RR}) accordingly:
$$
\sup_{x, y} \int_0^\infty t \bigg|\frac {\sin(t \sqrt H) P_{ac}}{\sqrt H}(x, y)\bigg| \dd t < \infty.
$$
This proves (\ref{one}). In particular,
\be\lb{temp}
\sup_{x, y} \int_t^\infty \bigg|\frac {\sin(\tau \sqrt H) P_{ac}}{\sqrt H}(x, y)\bigg| \dd \tau \les \frac 1 t.
\ee
Setting $t=|x-y|$ in (\ref{temp}),
\be\lb{part}
\sup_{x, y} \int_{|x-y|}^\infty \bigg|\frac {\sin(\tau \sqrt H) P_{ac}}{\sqrt H}(x, y)\bigg| \dd \tau \les \frac 1 {|x-y|}.
\ee
On the interval where $\tau<|x-y|$, due to the finite speed of propagation of solutions to the wave equation
$$
\frac {\sin(\tau \sqrt H)}{\sqrt H}(x, y) = 0.
$$
Hence on this interval
$$
\frac {\sin(\tau \sqrt H) P_{ac}}{\sqrt H}(x, y) = -\frac {\sin(\tau \sqrt H) P_p}{\sqrt H}(x, y).
$$
Since $H$ has finitely many negative eigenvalues and no embedded eigenvalues or resonances (partly due to Assumption \ref{a1}), the point component of the propagator can be expressed as
$$
\frac {\sin(\tau \sqrt H) P_p}{\sqrt H}(x, y) = \sum_{n=1}^N \frac {\sinh(\tau\lambda_n)}{\lambda_n} f_n(x) \otimes \ov f_n(y),
$$
where $f_n$ is the normalized $L^2$ eigenfunction that corresponds to the eigenvalue $-\lambda_n^2$, $1 \leq n \leq N$.

Focusing on just one term in this sum, by the Agmon bound Lemma \ref{expo}, for $\tau \in [0, < |x-y|)$
$$
\int_0^{|x-y|} \frac {\sinh(\tau\lambda_n)}{\lambda_n} f_n(x) \otimes \ov f_n(y) \les e^{|x-y|\lambda_n} \frac {e^{-\lambda_n|x|}} {\langle x \rangle} \frac {e^{-\lambda_n|y|}} {\langle y \rangle} \leq \langle x \rangle^{-1} \langle y \rangle^{-1} \les \frac 1 {|x-y|}.
$$
Thus estimate (\ref{part}) is thus also valid for the interval $[0, |x-y|)$, resulting in (\ref{two}).
\end{proof}

\begin{proof}[Proof of Theorem \ref{thm1}]
By functional calculus write
$$\begin{aligned}
e^{itH} P_{ac} &= \frac 1 {2\pi i} \int_0^\infty e^{it\lambda} (R(\lambda-i0)-R(\lambda+i0)) P_{ac} \dd \lambda \\
&= \frac i {\pi} \int_\R e^{it\lambda^2} R((\lambda+i0)^2) P_{ac} \lambda \dd \lambda \\
&= c\, t^{-1} \int_\R e^{it\lambda^2} \partial_\lambda R((\lambda+i0)^2) P_{ac} \dd \lambda \\
&= \tilde c\, t^{-1} \int_\R [\ov {\mc F e^{-it\lambda^2}}](\rho) \mc F [\partial_\lambda R((\lambda+i0)^2) P_{ac}](\rho) \dd \rho.
\end{aligned}$$
However, $\mc F [\partial_\lambda R((\lambda+i0)^2) P_{ac}]$ is absolutely integrable by (\ref{one}) and $\mc F e^{-it\lambda^2} \les t^{-1/2}$, meaning that $[e^{itH} P_{ac}](x, y) \les t^{-3/2}$, as claimed.

Also see \cite{BecGol1} for more details.
\end{proof}

\appendix
\section{Eigenvalues of the Hamiltonian}

Using the methods of \cite{BecGol4}, we prove that $H$ has only finitely many negative eigenvalues. We have already proved that there are no embedded eigenvalues, in the proof of Proposition \ref{prop_wiener}, using the results of \cite{KocTat}, under more assumptions on $A$ and $V$ than necessary. The proof of a sharper result is left for a future paper.


Let $\K_0$ be the closure of the class of test functions $\mc D$ in $\K$ and let $\K_{20}$ be the closure of $\mc D$ in $\KK$. We assume that $A \in \K_{20} \cap \K_0$ and $V \in \K_0$.


The proof is based on Fechbach's lemma from \cite{Yaj}:
\begin{lemma}\lb{fechbach} Let $X=X_0 \oplus X_1$ be a direct sum decomposition of a vector space $X$ and for $L \in \B(X)$ consider its matrix decomposition
$$
L=\begin{pmatrix}
L_{00} & L_{01}\\
L_{10} & L_{11}
\end{pmatrix},
$$
where $L_{ij} \in \B(X_j, X_i)$. Assume $L_{00}$ is invertible. Then $L$ is invertible if and only if $C=L_{11}-L_{10} L_{00}^{-1} L_{01}$ is invertible, in which case
$$
L^{-1}=\begin{pmatrix}
L_{00}^{-1} + L_{00}^{-1} L_{01} C^{-1} L_{10} L_{00}^{-1} & -L_{00}^{-1} L_{01} C^{-1}\\
-C^{-1} L_{10} L_{00}^{-1} & C^{-1}
\end{pmatrix}.
$$
\end{lemma}

%
%


We also use the Gohberg--Segal version \cite{GohSig} of Rouch\'{e}'s theorem for families of operators, stated here for the case of analytic (as opposed to meromorphic) functions, for simplicity:
\begin{theorem}\lb{rouche} Let $O \subset \C$ be a simply connected open region with rectifiable boundary $\partial O$ and let $A:\ov O \to \B(X, \tilde X)$ be a family of operators between two Banach spaces, analytic and Fredholm ($\dim \ker A(z)<\infty$, $A(z)(X)$ is closed in $\tilde X$, and $\dim \tilde X/A(z)(X)< \infty$) at each point $z \in O$ and continuous on $\ov O$, such that $A(z)$ is invertible for $z \in \partial O$. Consider another family of operators $B: \ov O \to \B(X, \tilde X)$, analytic in $O$ and continuous on $\ov O$, such that $\sup_{z \in \partial O} \|A^{-1}(z)B(z)\|_{\B(X)} <1$. Then $A$ and $A+B$ have the same number of zeros in $O$.
\end{theorem}

Let $H_t = -\Delta + U_t = -\Delta + A_t \nabla + V_t$, where $U_t = tU$, $V_t=tV$, $A_t = tA$.

\begin{proposition}[see Proposition B.6 in \cite{BecGol4}]\lb{lowerbound} If $A \in \K_{20} \cap \K_0$ and $V \in \K_0$, the spectrum of $H_t$ is bounded from below, uniformly for $t \in [0, 1]$.
\end{proposition}
\begin{proof} For test functions $V$ and $A$
$$
\lim_{a \to +\infty} \sup_x \int_{\R^3} \frac {e^{-a|x-y|}|V(y)|}{|x-y|} = 0,\ \lim_{a \to +\infty} \sup_x \int_{\R^3} \frac {a e^{-a|x-y|}|A(y)|}{|x-y|} = 0,\ \lim_{a \to +\infty} \sup_x \int_{\R^3} \frac {e^{-a|x-y|}|A(y)|}{|x-y|^2} = 0.
$$
By approximating, the same will be true for $A \in \K_{20} \cap \K_0$ and $V \in \K_0$. Obviously, this is the case uniformly for $V_t$ and $A_t$, when $t \in [0, 1]$. Then we can invert $I+R_0(-\lambda^2)U$ by a Born series for any sufficiently large $\lambda$, meaning that such points are in the resolvent set, in view of the identity (\ref{res_id})
$$
R=(I+R_0U)^{-1}R_0.
$$
\end{proof}
Unlike in \cite{BecGol4}, here the proof extends to any angle less than $\pi$, but not to a whole half-plane, because of the $ae^{-a|x-y|}$ term.

\begin{proposition}[see Proposition B.7 in \cite{BecGol4}] If $A \in \K_{20} \cap \K_0$, $\nabla A \in \K_0$, $V \in \K_0$ and $U=A \nabla + V$ is self-adjoint, then $H=-\Delta+U$ has the same number of negative eigenvalues as the number of eigenvalues (with multiplicity) of $U (-\Delta)^{-1} \in \B(L^1 \cap \K \cap \dot H^{-1})$ in $(-\infty, -1)$, which is finite.
\end{proposition}

As in \cite{BecGol4}, in the course of the proof we also show that $H_t=-\Delta+U_t$ has a threshold eigenvalue or resonance for at most discretely many values of $t$. In this sense, Assumption \ref{a1} is generic.
\begin{proof}
Let $r_0$ be such that, by Proposition \ref{lowerbound}, $H_t$ has no eigenvalues below $-r_0$ for $t \in [0, 1]$.
	
Consider the contour $\mc C=\partial B(-r_0, r_0)$ (the circle in the complex plane centered at $-r_0$ of radius $r_0$) and let
$$
B_t(\lambda) = I+U_tR_0(\lambda).
$$
By Rouch\'e's theorem for analytic families of operators, any $t_0$ for which $B_{t_0}(\lambda)$ has no zeros on $\mc C$ has a neighborhood in $\R$ on which the number of zeros of $B_t$ inside $\mc C$ is constant. The maximal such sets are intervals.
	
Due to the self-adjointness of $-\Delta+U$ and to the lower bound on the spectrum, zeros of $H_t$, $t \in [0, 1]$, on $\mc C$ can only occur at $0$. This means there exists some $f \in L^1 \cap \K \cap \dot H^{-1}$, $f \ne 0$, such that
$$
f=-U_tR_0(0)f=-tU(-\Delta)^{-1}f.
$$
Then $t \ne 0$ and $-1/t$ is an eigenvalue of $U(-\Delta)^{-1}$ in $(-\infty, -1]$.
	
Since $U(-\Delta)^{-1}$ is compact, it has only finitely many eigenvalues outside any positive radius around the origin. Thus the interval $[0, 1]$ is divided into finitely many subintervals on which the number of zeros of $H_t$ inside $\mc C$ is constant, separated by finitely many values $t_k$ for which $B_{t_k}$ has a zero at $0$.
	
Next, we consider the effect of small perturbations on the zero of $B_{t_k}$ at $0$. With no loss of generality, take $t_k=1$. Let $N=\dim \Ker B_1(0)$ be the multiplicity of this zero, which equals the multiplicity of $-1/t_k$ as an eigenvalue of $UR_0(0)$.
	
Let $P_0$ be the projection on $\Ker B_1(0)$, given by $P_0 f = \sum_{k=1}^N \langle f, g_k \rangle f_k$, where for $1 \leq k \leq N$ $f_k = -U R_0(0) f_k$, $f_k = - U g_k$, $g_k = R_0(0) f_k$, $g_k = -R_0(0) U g_k$. The functions $g_k$ are the bound states of $H$ at energy $0$.

It is easy to see that $g_k \in L^\infty \cap \K^* \cap \dot H^1$. Furthermore, $g_k \in \langle x \rangle^{-1} L^\infty$. Indeed, by assumption, one can alternatively write $U=\nabla A + \tilde V$, where $\tilde V = V-(\nabla A) \in \K_0$. Write the equation $g_k=-R_0(0)Ug_k$ as
$$
g_k = -R_0(0) U_1 g_k - R_0(0) U_2 g_k
$$
where $U_1=(\nabla A + \tilde V) \chi_{>R}(x)$ and $U_2=(\nabla A + \tilde V) \chi_{\leq R}(x)$ for some sufficiently large $R$. Then $I+R_0(0) U_1$ is invertible on $\langle x \rangle^{-1} L^\infty$ and $R_0(0) U_2 g_k \in \langle x \rangle^{-1} L^\infty$.

Note that $B_1(0)$ is invertible on $X_1=P_1 (L^1 \cap \K)$, hence so is $P_1 B_1(\lambda) P_1$ for $\lambda$ sufficiently close to $0$, with uniformly bounded inverse in $\B(L^1) \cap \B(\K)$.

We first prove that $\lambda=0$ is an isolated zero for $B_1(\lambda)$. By Fechbach's Lemma \ref{fechbach}, for $\lambda$ close to $0$, $B_1(\lambda)$ is invertible if and only if
\be\lb{expr11}
D_1(\lambda) = P_0 B_1(\lambda) P_0 - P_0 B_1(\lambda) P_1 (P_1 B_1(\lambda) P_1)^{-1} P_1 B_1(\lambda) P_0
\ee
is.

Next, suppose that (\ref{expr11}) is not invertible. Since (\ref{expr11}) is defined on a finite-dimensional vector space, this means there exists $f \in \Ker B_1(0)$, $f \ne 0$, $f=P_0 f$, such that
\be\lb{expr2}
\langle P_0 B_1(\lambda) P_0 f, g \rangle = \langle P_0 B_1(\lambda) P_1 (P_1 B_1(\lambda) P_1)^{-1} P_1 B_1(\lambda) P_0 f, g \rangle,
\ee
where $f=Ug$ is paired with $g=R_0(0)f$.

Since $P_0 B_1(0) = B_1(0) P_0 = 0$, $P_0 B_1(\lambda) = P_0 [B_1(\lambda)-B_1(0)]$, and $B_1(\lambda) P_0 = [B_1(\lambda)-B_1(0)] P_0$, rewrite the right-hand side as
$$
\langle P_0 [B_1(\lambda)-B_1(0)] P_1 [P_1 B_1(\lambda) P_1]^{-1} P_1 [B_1(\lambda)-B_1(0)] P_0 f, g \rangle.
$$
Note that
$$
\|R_0(\lambda)-R_0(0)\|_{\B(L^1, L^\infty)} \les \lambda^{1/2},\ \|R_0(\lambda)-R_0(0)\|_{\B(L^1, \K^*) \cap \B(\K, L^\infty)} \les 1,
$$
while (see (\ref{derivative}))
$$
\|\nabla R_0(\lambda)-\nabla R_0(0)\|_{\B(L^1, \K^*)} \les \lambda^{1/2},\ \|\nabla R_0(\lambda) - \nabla R_0(0)\|_{\B(L^1, \K^*+\K_2^*)} \les 1
$$
as well. By approximating $V$ with potentials in $\K \cap L^1$, we obtain
\be\lb{app1}
\|P_0 (B_1(\lambda)-B_1(0)) P_1 (P_1 B_1(\lambda) P_1)^{-1} P_1 (B_1(\lambda)-B_1(0)) P_0 f\| \les o(\lambda^{1/2}) \|f\|_{L^1}
\ee
as $\lambda \to 0$. At the same time, the left-hand side in (\ref{expr2}) admits the asymptotic expansion
$$
\langle P_0 B_1(\lambda) P_0 f, g \rangle = \langle (R_0(\lambda)-R_0(0)) f, f \rangle = c\lambda^{1/2} |\langle f, 1 \rangle|^2 + o(\lambda^{1/2}).
$$
Therefore $\langle \tilde V g, 1 \rangle = \langle \tilde V g + \nabla (A g), 1 \rangle= \langle f, 1\rangle=0$ and the corresponding $g=R_0(0)f \in L^2$ is an eigenfunction:
$$
g(x)=-\frac 1 {4\pi} \int_{\R^3} \bigg(\frac {1}{|x-y|} \tilde V(y) - \frac {x-y}{|x-y|^3} \bigg) g(y) \dd y = -\frac 1 {4\pi} \int_{\R^3} \bigg( \bigg(\frac {1}{|x-y|} - \frac 1 {|x|}\bigg) \tilde V(y) - \frac {x-y}{|x-y|^3} \cdot A(y) \bigg) g(y) \dd y.
$$
Likewise
$$
f(x) = -\frac 1 {4\pi} \int_{\R^3} \bigg(V(x) \bigg(\frac 1 {|x-y|}-\frac 1 {|x|}\bigg) - A(x) \frac {x-y}{|x-y|^3}\bigg) f(y) \dd y.
$$
Then $g \in \langle x \rangle^{-2} L^\infty$ and $f \in \langle x \rangle^{-1} L^1 \cap \langle x \rangle^{-2} \K$. The left-hand side of (\ref{expr2}) now has the asymptotic expansion
$$
c\lambda \int_{\R^3} \int_{\R^3} f(x) |x-y| f(y) \dd x \dd y + o(\lambda).
$$
For the right-hand side, note that
\be\lb{rhs1}
\frac {e^{i\sqrt \lambda |x-y|} - e^{i\sqrt\lambda|x|}}{|x|} \les \lambda^{1/2} \frac {|y|}{|x|},\ \bigg\|\int_{\R^3} \frac {e^{i\sqrt \lambda |x-y|} - e^{i\sqrt\lambda|x|}}{|x|} f(y) \dd y\bigg\|_{|x|^{-1} L^\infty} \les \lambda^{1/2} \|f\|_{\langle x \rangle^{-1} L^1}.
\ee
Since $\langle f, 1\rangle = \langle V g, 1 \rangle=0$, one can subtract the term
$$
\frac 1 {4\pi} \int_{\R^3} \frac {e^{i\sqrt\lambda|x|}-1}{|x|} f(y) \dd y = 0
$$
from
$$
[R_0(\lambda)-R_0(0)] f = \frac 1 {4\pi} \int_{\R^3} \frac {e^{i\sqrt\lambda|x-y|}-1}{|x-y|} f(y) \dd y
$$
then in light of (\ref{rhs1}) we can replace this term by $\frac {e^{i\sqrt\lambda|x-y|}-1}{|x|} f(y)$. Furthermore
$$
\frac {e^{i\sqrt\lambda|x-y|}-1}{|x-y|} - \frac {e^{i\sqrt\lambda|x-y|}-1}{|x|} \les \lambda^{1/2} \frac {|y|}{|x|},\ \bigg\|\int_{\R^3} \bigg(\frac {e^{i\sqrt\lambda|x-y|}-1}{|x-y|} - \frac {e^{i\sqrt\lambda|x-y|}-1}{|x|}\bigg) f(y) \dd y\bigg\|_{|x|^{-1} L^\infty} \les \lambda^{1/2} \|f\|_{\langle x \rangle^{-1} L^1}.
$$
For the gradient term (\ref{derivative}) we simply use the fact that
$$
\frac {(e^{i\sqrt \lambda|x-y|}-1)(x-y)}{|x-y|^3} \les \lambda^{1/2} \frac 1 {|x-y|^2}.
$$
Consequently
$$
\|(B_1(\lambda)-B_1(0)) P_0 f\|_{L^1} \les \lambda^{1/2} \|f\|_{\langle x \rangle^{-1} L^1}
$$
and for the right-hand side in (\ref{expr2}) we obtain
$$
\|P_0 (B_1(\lambda)-B_1(0)) P_1 (P_1 B_1(\lambda) P_1)^{-1} P_1 (B_1(\lambda)-B_1(0)) P_0 f\| \les \lambda \|f\|_{\langle x \rangle^{-1} L^1}.
$$
If $V \in \langle x \rangle^{-1} L^1$ and $A \in L^1$, since $\langle f, 1 \rangle = 0$ and
$$
\frac {e^{i\sqrt \lambda |x-y|}-1-i\sqrt\lambda|x-y|}{|x-y|} \les \lambda |x-y|,\ \frac {(e^{i\sqrt \lambda|x-y|}-1-i\sqrt \lambda |x-y| e^{i\sqrt\lambda|x-y|})(x-y)}{|x-y|^3} \les \lambda,
$$
a better estimate is
$$
\|(B_1(\lambda)-B_1(0)) P_0 f\|_{L^1} \les \lambda \|f\|_{\langle x \rangle^{-1} L^1}.
$$
Consequently, approximating $V$ and $A$ with potentials in this better class we obtain
\be\lb{app2}
\|P_0 (B_1(\lambda)-B_1(0)) P_1 (P_1 B_1(\lambda) P_1)^{-1} P_1 (B_1(\lambda)-B_1(0)) P_0 f\| \les o(\lambda) \|f\|_{\langle x \rangle^{-1} L^1}.
\ee

This leads to a contradiction in (\ref{expr2}): $\int_{\R^3} \int_{\R^3} f(x)|x-y|f(y) \dd x \dd y = 0$ and since $\langle f, 1 \rangle = 0$ this means $\|f\|_{\dot H^{-1}} = 0$, hence $f=0$, which contradicts our original assumption that $f \ne 0$.

We now examine how the number of bound states changes in the neighborhood of a transition point. Without loss of generality, let $t_n=1$. Then $B_t(\lambda)$ for $|t-1|=\epsilon<<1$ and $|\lambda|<\epsilon^2$ ($t$ close to $1$ and $\lambda$ close to $0$) is invertible if and only if
\be\lb{full}
P_0 B_t(\lambda) P_0 - P_0 B_t(\lambda) P_1 (P_1 B_t(\lambda) P_1)^{-1} P_1 B_t(\lambda) P_0
\ee
is invertible on $\Ker B_1(0)$. Recall $P_0$ is the projection on $\Ker B_1(0)$ and $P_1 = I-P_0$.

For each $|t|<1$ $P_0 B_t(\lambda) P_0 = P_0 [B_t(\lambda)-B_1(0)] P_0$ is, uniformly in $\lambda$, non-degenerate, because it is bounded from below by some multiple of the $\dot H^{-1}$ dot product. Its inverse is of size $|t-1|^{-1}$. Since the second term in (\ref{full}) is quadratic, it is of size $o(t-1)$ as $t$ approaches $1$. By Rouch\'{e}'s theorem, when $t<1$ is close to $1$, $B_t$ has the same number of negative zeros as $B_1$ in a small punctured neighborhood of $\lambda=0$, namely no zeros.

A deeper analysis is required for $t>1$. Further decompose $\Ker B_1(0)$ by considering the subspace of functions $f$ orthogonal to $1$, which correspond to eigenstates: $f=-Ug$ where $g \in L^2$ is an eigenstate. Then take the orthogonal complement of this subspace with respect to $(-\Delta)^{-1}$, which defines a dot product on $\Ker B_1(0)$. The orthogonal complement has dimension at most $1$, consisting of the canonical resonance, if present.

Let $P_0=Q_1+Q_2$, where $Q_1 f = \langle f, g_1 \rangle f_1$ corresponds to the canonical resonance and $Q_2 f = \sum_{k=2}^N \langle f, g_k \rangle f_k$ to the eigenstates. We only examine the most complicated case, when $Q_1 \ne 0$ and $Q_2 \ne 0$.

Let $\lambda=-k^2 \leq 0$. We estimate the second term in (\ref{full}). In particular, we need to deal with factors of $B_t(0)-B_1(0)$. Note that $B_t(0)-B_1(0)=(t-1)V R_0(0)$, so $P_0 [B_t(0)-B_1(0)] = [B_t(0)-B_1(0)] P_0 = (1-t) P_0$. Due to the presence of $P_1$, the contribution of such factors to the second term in (\ref{full}) cancels.

Due to estimates (\ref{app1}) and (\ref{app2}), the second term in (\ref{full}) is then of size
\be\lb{err}
E=P_0 (o(k) Q_1 + o(k^2) Q_2).
\ee

Consider the Taylor expansions of $(-\Delta-\lambda)^{-1}$
$$
(-\Delta+k^2)^{-1}(x, y) = \frac 1 {4\pi} \frac {e^{i\sqrt \lambda |x-y|}}{|x-y|} = \frac 1 {4\pi} \bigg(\frac 1 {|x-y|} - k 1 + o(k)\bigg) = \frac 1 {4\pi} \bigg(\frac 1 {|x-y|} - k 1 + \frac {k^2} 2 |x-y| + o(k^2|x-y|)\bigg)
$$
and $\nabla (-\Delta-\lambda)^{-1}$ (see (\ref{derivative}))
$$
\nabla (-\Delta+k^2)^{-1}(x, y) = \frac 1 {4\pi} \bigg(-\frac {ke^{-k|x-y|}}{|x-y|} - \frac {e^{-k|x-y|}}{|x-y|^2}\bigg) \frac {x-y}{|x-y|} = \frac 1 {4\pi} \bigg(-\frac {x-y}{|x-y|^3} + \frac {k^2} 2 \frac {x-y}{|x-y|} + o(k^2)\bigg).
$$
Here $1$ is a rank-one non-negative operator and $\ds\frac 1 {4\pi} |x-y|$ is a negative operator on the space of $\langle x \rangle^{-1} L^1$ functions orthogonal to $1$ (i.e.~whose integral is zero), where it is the integral kernel of $-(-\Delta)^{-2}$, while $\ds \frac 1 {4\pi} \frac{x-y}{|x-y|}$ is its gradient with respect to $x$.

This produces the following asymptotic expansion of the first term in (\ref{full}) $P_0 B_t(\lambda) P_0 = P_0 [B_t(\lambda)-B_1(0)] P_0$:
$$\begin{aligned}
P_0 B_t(\lambda) P_0 &= \frac 1 {4\pi} \bigg((t-1) P_0 \Big(V \frac 1 {|x-y|} - A \cdot \frac {x-y}{|x-y|^3}\Big) P_0 - (tk) Q_1 V 1 Q_1 + Q_1 o(k) Q_1 + \frac {tk^2} 2 Q_2 \Big(V |x-y| + A \cdot \frac {x-y}{|x-y|}\Big) Q_2 +\\
&+ t Q_1 V \frac {e^{-k|x-y|}-1+k|x-y|}{|x-y|} Q_2 + t Q_2 V \frac {e^{-k|x-y|}-1+k|x-y|}{|x-y|} Q_1 + Q_2 o(k^2) Q_2 + P_0 o(k^2) P_0\bigg).
\end{aligned}$$

The first-order term in the expansion has no cotribution from $A$, because $\nabla 1 = 0$.

We have written explicitly the most problematic two error terms. To bound them, we have to evaluate the integrals
\be\lb{ex}
\int_{\R^3} \int_{\R^3} f_1(x) \frac {e^{-k|x-y|}-1+k|x-y|}{|x-y|} f_j(y) \dd x \dd y,
\ee
where $2 \leq j \leq N$. Note that in the integral
$$
\int_{\R^3} \int_{\R^3} f_1(x) |x-y| f_j(y) \dd x \dd y
$$
one can replace $|x-y|$ by $|x-y|-|x|$, since $\langle f_j, 1 \rangle = 0$. Then $||x-y|-|x||\leq |y|$ and since $f_j \in \langle x \rangle^{-1} L^1$ the integral is well-defined, but not absolutely convergent.

Similarly, let $F(t)=\frac {e^{-kt}-1+kt}{t}$; then $F(|x-y|)-F(|x|) \les k^2 |y|^2$. As before, we can subtract $F(|x|)$ in (\ref{ex}) for free, since $\langle f_j, 1 \rangle=0$, and then the integral of the difference is of size $O(k^2)$, since $f_j \in \langle x \rangle^{-2} \K$.

Grouping most of the error terms together and noting that $Q_1 o(k) Q_1$, $Q_2 O(k^2) Q_1$, $Q_2 o(k^2) Q_2$, and $P_0 o(k^2) P_0$ satisfy the bound (\ref{err}), we then approximate (\ref{full}) and its first term $P_0(B_t(\lambda)-B_1(0))P_0$ by
$$\begin{aligned}
D_{t, \lambda} &= \frac 1 {4\pi} \bigg((t-1) P_0 U \frac 1 {|x-y|} P_0 - (t k) Q_1 U 1 Q_1 + \frac {t k^2} 2 Q_2 U |x-y| Q_2 + Q_1 O(k^2) Q_2 \bigg) \\
&= (t-1) D_0 - tk D_1 - tk^2 D_2 + Q_1 O(k^2) Q_2 \\
&= Q_1 [(t-1)I-tkD_1] Q_1 + Q_2 [(t-1)I - tk^2 D_2 ] Q_2 + Q_1 O(k^2) Q_2.
\end{aligned}$$
Here $D_0$ is the matrix representation of $\ds(-\Delta)^{-1} = \frac 1 {4\pi|x-y|}$ as a quadratic form on $\Ker B_1(0)$, which we can take to be the identity matrix by choosing an appropriate basis, while $D_1=Q_1 U D_1Q_1$ is the matrix representation of $\ds P_0 U \frac 1 {4\pi} P_0$ and $D_2=Q_2 U D_2 Q_2$ is the matrix representation of $\ds-\frac 1 {8\pi} Q_2 U |x-y| Q_2 = \frac 1 2 (-\Delta)^{-2}$ on $Q_2 \Ker B_1(0)$.

One can represent $D_{t, \lambda}$ as a block matrix:
$$
D_{t, \lambda} = \begin{pmatrix}
	(t-1)I-tkD_1 & O(k^2) \\
	0 & (t-1)I-tk^2 D_2
\end{pmatrix} = \begin{pmatrix} D_{11} & D_{12} \\ 0 & D_{22} \end{pmatrix}.
$$

We first examine the invertibility of the diagonal entries. Note that $D_1$ has exactly one eigenvalue $\lambda_1$ and $D_2$ has $N-1$ eigenvalues $\lambda_2, \ldots, \lambda_N$, all positive, since $D_1$ and $D_2$ are hermitian and positive definite.

For $t>1$ close to $1$, $D_{t, \lambda}$ has exactly $N$ zeros, one for $\ds \frac {t-1}{tk} = \lambda_1$ and $N-1$ zeros for $\ds \frac {t-1}{tk^2} = \lambda_j$, $2 \leq j \leq N$. In terms of $t-1=\epsilon<<1$, the zeros are at $k \sim \epsilon$ and $k \sim \epsilon^{1/2}$, meaning $\lambda \sim -\epsilon^2$ and $\lambda \sim -\epsilon$.

The inverse of $D_{t, \lambda}$ is
$$
D_{t, \lambda}^{-1} = \begin{pmatrix} D_{11}^{-1} & -D_{11}^{-1} D_{12} D_{22}^{-1} \\ 0 & D_{22}^{-1} \end{pmatrix}.
$$
When $D_{t, \lambda}$ is invertible, the inverses of the diagonal terms are of size
$$
Q_1 O((t-1-tk\lambda_1)^{-1}) Q_1
$$
and
$$
Q_2 O(\sup_{2 \leq j \leq N} |t-1-tk^2\lambda_j|^{-1}) Q_2
$$
and the off-diagonal term in the inverse has size
$$
Q_1 O((t-1-tk\lambda_1)^{-1}) O(k^2) O(\sup_{2 \leq j \leq N} |t-1-tk^2\lambda_j|^{-1}) Q_2.
$$

Consider a contour (a small square, for example) that cuts the real axis at $\lambda=0$ and $\lambda=-k^2=-C\epsilon$. For $C$ fixed and sufficiently large, all $N$ zeros are inside the contour.

On the contour, near $\lambda=0$ (for $|\lambda|<\epsilon^2$), the inverses of the diagonal terms are of size $\epsilon^{-1}$ and the off-diagonal term is of size $1$. Consequently the product between the error term $E$ (\ref{err}) and $D_{t, \lambda}^{-1}$ is of size $E D_{t, \lambda}^{-1} \les o(1)$.

Away from $0$ the inverses of the diagonal terms are of size $Q_1 O(k^{-1}) Q_1$ and $Q_2 O(\epsilon^{-1}) Q_2$, while the off-diagonal term is of size $Q_1 O(k \epsilon^{-1}) Q_2 = Q_1 O(k^{-1}) Q_2$, as $k\epsilon^{-1} \les k^{-1}$. Since $\epsilon^{-1} \les k^{-2}$, again $E D_{t, \lambda}^{-1} \les o(1)$.

By Rouch\'e's Theorem \ref{rouche}, then, for sufficiently small $\epsilon$, the full expression (\ref{full}) has the same number $N$ of zeros inside the contour as the approximation $D_{t, \lambda}$.

Thus, the number of negative eigenvalues of $H_t$ grows by no more than the number of zeros at each transition point $t_n$, considered with multiplicity.
Since this happens at all transition points and there are finitely many of them, the conclusion follows.
\end{proof}

Finally, we prove the Agmon bound for the eigenfunctions of $H$.
\begin{lemma}\lb{expo} Suppose that $H=-\Delta+U$, $U=\nabla A + V$, is such that $A \in \K_{2, 0} \cap \K_0$ and $V \in \K_0$ and let $Hf = \lambda f$, $\lambda<0$, $f \in L^2$, be an eigenfunction of $H$. Then $e^{-\lambda|x|} \langle x \rangle f \in L^\infty$.
\end{lemma}
This is similar to such results in the scalar potential case, see \cite{BecChe} for example.
\begin{proof}
The equation of the eigenfunction is $f = - R_0(-\lambda^2) U f$, meaning
$$
f(x) = - \frac {1}{4\pi} \int_{\R^3} \frac {e^{-\lambda|x-y|}}{|x-y|} V(y) f(y) \dd y - \frac {1}{4\pi} \int_{\R^3} \bigg(\frac {-\lambda e^{-\lambda|x-y|}}{|x-y|} - \frac {e^{-\lambda|x-y|}}{|x-y|^2} \bigg) \frac {x-y}{|x-y|} \cdot A(y) f(y) \dd y,
$$
where $f \in L^2$.

Split $U=U_0+\tilde U$, where $U_0=\nabla A_0+V_0$, $\tilde U=\nabla \tilde A + \tilde V$, $A=A_0+\tilde A$, $V=V_0+\tilde V$, where $A_0$ and $V_0$ are bounded and compactly supported and $\tilde A$ and $\tilde V$ are small in norm. Now write
$$
f = - R_0(-\lambda^2) \tilde U f - R_0(-\lambda^2) U_0 f,
$$
so, since $I+R_0(-\lambda^2) \tilde U$ is invertible (by a power series, since $\tilde U$ is small),
$$
f = (I+R_0(-\lambda^2) \tilde U)^{-1} R_0(-\lambda^2) U_0 f.
$$
As $R_0(-\lambda^2) U_0 f \in e^{\lambda |x|} \langle x \rangle^{-1} L^\infty$ and $R_0(-\lambda^2) \tilde U$ is bounded and small in $\B(e^{\lambda |x|} \langle x \rangle^{-1} L^\infty)$, the conclusion follows.
\end{proof}

\end{document}